\documentclass{article}
\usepackage{amsfonts,amsmath,latexsym,hyperref}
\usepackage{color}
\newcounter{examplenum}
\newtheorem{example}{Example}[examplenum]
\newcounter{lemmanum}
\newtheorem{lemma}{Lemma}[lemmanum]
\newcounter{propnum}
\newtheorem{proposition}{Proposition}[propnum]
\newcounter{satznum}
\newtheorem{theorem}{Theorem}[satznum]
\newenvironment{remark}
{\begin{trivlist}\item[]{\bf Remark.}}
	{\end{trivlist}}
\newenvironment{remarks}
{\begin{trivlist}\item[]{\bf Remarks.}}
	{\end{trivlist}}

\newenvironment{proof}
{\begin{trivlist}\item[]{\bf Proof.}}
	{\end{trivlist}}
{\makeatletter
	\gdef\cz{{\mathbb C}} 
	\gdef\me{{\mathbb E}} 
	\gdef\nz{{\mathbb N}} 
	\gdef\pr{{\mathbb P}} 
	\gdef\rz{{\mathbb R}} 
}
\newcounter{todocounter}

\makeatletter
\def\@MRExtract#1 #2!{#1}
\newcommand{\MR}[1]{
	\xdef\@MRSTRIP{\@MRExtract#1 !}
	\href{http://www.ams.org/mathscinet-getitem?mr=\@MRSTRIP}{MR\@MRSTRIP}}
\makeatother

\begin{document}
    \section*{Scaling limits for a class of regular $\Xi$-coalescents}
	{\sc Martin M\"ohle} and {\sc Benedict Vetter}\footnote{Mathematisches Institut, Eberhard Karls Universit\"at T\"ubingen, Auf der Morgenstelle 10, 72076 T\"ubingen, Germany, E-mail addresses: martin.moehle@uni-tuebingen.de, benedict.vetter@uni-tuebingen.de}
	\begin{center}
		\today
	\end{center}
	\begin{abstract}
\noindent The block counting process with initial state $n$ counts the number of blocks of an exchangeable coalescent ($\Xi$-coalescent) restricted to a sample of size $n$. This work provides scaling limits for the block counting process of regular $\Xi$-coalescents that stay infinite, including $\Xi$-coalescents with dust and a large class of dust-free $\Xi$-coalescents. The main convergence result states that the block counting process, properly logarithmically scaled, converges in the Skorohod space to an Ornstein--Uhlenbeck type process as $n$ tends to infinity. The existence of such a scaling depends on a sort of curvature condition of a particular function well-known from the literature. This curvature condition is intrinsically related to the behavior of the measure $\Xi$ near the origin. The method of proof is to show the uniform convergence of the associated generators. Via Siegmund duality an analogous result for the fixation line is proven. Several examples are studied.

\vspace{2mm}		

		\noindent Keywords: Block counting process; fixation line; Ornstein--Uhlenbeck type process; regular coalescent; simultaneous multiple collisions; time-inhomogeneous process; weak convergence
		
		\vspace{2mm}
		
		\noindent 2020 Mathematics Subject Classification:
		Primary
		60J90 
		Secondary
		60J27 
	\end{abstract}
	
\subsection{Introduction}

Exchangeable coalescents are continuous-time Markov processes taking values in the space $\mathcal{P}$ of partitions of $\nz:=\{1,2,\ldots\}$, where blocks merge over time. Their distribution is determined by a finite measure $\Xi$ on the infinite simplex $\Delta:=\{(u_1,u_2,\ldots):u_1\ge u_2\ge\cdots\ge0,\sum_{i\ge 1} u_i\le1\}$. Coalescents can be constructed from appropriate Poisson point processes (Schweinsberg \cite{schweinsberg00}), which allows to identify the class of exchangeable coalescents with the class of finite measures $\Xi$ on $\Delta$. In the Cannings model \cite{cannings74, cannings75}, a discrete-time haploid population model with nonoverlapping generations and finite constant population size, individuals of the same generation follow an exchangeable reproduction law, independently of the other generations. Start with a sample of individuals in one generation and put members into the same block, when they have a common parent one generation in the past. We obtain a discrete-time partition-valued ancestral process by merging individuals who share a common ancestor when going backwards further in time, and the coagulation of ancestral lineages corresponds to the merging of blocks. Under suitable conditions exchangeable coalescents then arise as the weak limit of these ancestral processes, properly time-scaled, as the total population size tends to infinity, since a certain form of consistency relation holds for Cannings models, see \cite{moehlesagitov01}.

Most coalescents treated in the literature belong to one of the following subclasses. The coalescent $(\Pi_t)_{t\ge0}$, starting from an infinite number of blocks, is said to come down from infinity if the number of blocks is finite at all times $t>0$ almost surely, and it is said to stay infinite if the number of blocks is infinite at all times $t>0$ almost surely. For coalescents with dust the number of original blocks that have not been involved in any merger up to time $t>0$ is infinite with positive probability. Schweinsberg \cite{schweinsberg00} determined conditions to decide on the schemes.

Let $\Lambda$ be a finite measure on the unit interval $[0,1]$. The $\Lambda$-coalescent, which allows only for multiple but not for simultaneous multiple mergers of ancestral lineages, is the particular $\Xi$-coalescent, where the measure $\Xi$ on $\Delta$ is concentrated on $[0,1]\times\{0\}\times\{0\}\cdots$ with
$\Xi(B\times\{0\}\times\{0\}\times\cdots):=\Lambda(B)$ for all Borel sets $B\subseteq[0,1]$.

Suppose that $(\Pi_t)_{t\ge0}$ is standard, i.e., $\Pi_0$ is the partition of $\nz$ into singletons.
For $t\ge0$ and $n\in\nz$ the restriction $\Pi_t^{(n)}:=\{B\cap[n]:B\in\Pi_t,B\cap[n]\ne\emptyset\}$ of $\Pi_t$ to $[n]:=\{1,\ldots,n\}$ has values in the space $\mathcal{P}_n$
of partitions of $[n]$. Suppose that $\Pi^{(n)}:=(\Pi_t^{(n)})_{t\ge0}$ is in a state with $k\in[n]$ blocks. For $j\ge1$, $k_1\ge\cdots\ge k_j$ with $k_1+\cdots+k_j=k$ and $k_1\ge2$ we speak of a $(k_1,\ldots,k_j)$-collision, when $\Pi^{(n)}$ jumps to a state with $j$ blocks and $k_1,\ldots,k_j$ blocks merge into single blocks, respectively. 
We next introduce some standard notation. Define $|u|:=\sum_{i\ge 1} u_i$ and $(u,u):=\sum_{i\ge 1} u_i^2$ for $u\in\Delta$, $0:=(0,0,\ldots)\in\Delta$, $a:=\Xi(\{0\})$, and the measures $\Xi_0$ and $\nu$ via $\Xi=a\varepsilon_{0}+\Xi_0$ and $\nu({\rm d}u):=\Xi_0({\rm d}u)/(u,u)$. 
A $(k_1,\ldots,k_j)$-collision, $j\in\nz$, occurs at the rate (Schweinsberg \cite{schweinsberg00})
\[ \phi_j(k_1,\ldots,k_j)\ =\ a1_{\{j=1,k_1=2\}}\ +\ \int_{\Delta}\sum_{l=0}^{s}\binom{s}{l}(1-|u|)^{s-l}\sum_{i_1\ne\cdots\ne i_{r+l}}u_{i_1}^{k_1}\cdots u_{i_{r+l}}^{k_{r+l}}\,\nu({\rm d}u),
\]
where $s:=|\{i\in[j]:k_i=1\}|$, $r:=j-s$ and $k_1\ge\cdots\ge k_j$ with $k_1\ge2$.

The aim of this work is to analyze the block counting process $N^{(n)}:=(N_t^{(n)})_{t\ge0}:=(|\Pi_t^{(n)}|)_{t\ge0}$ for large initial state, more precisely, to determine scaling functions $v(n,t)$ for which $N_t^{(n)}/v(n,t)$ converges in distribution as $n\to\infty$. For coalescents with dust it is proven in \cite{gaisermoehle16} and \cite{moehle21} with different methods that $(N_t^{(n)}/n)_{t\ge0}$ converges in the space $D_{[0,1]}[0,\infty)$ of c\`adl\`ag paths endowed with the Skorohod topology to the so-called frequency of singletons process as $n\to\infty$. The Bolthausen--Sznitman coalescent in which the driving measure $\Lambda$ is the uniform distribution on $[0,1]$ has been thoroughly studied in the literature and is an example of a dust-free $\Lambda$-coalescent that stays infinite. Goldschmidt and Martin \cite{goldschmidtmartin05} and Baur and Bertoin \cite{baurbertoin15} proved for every $t\ge0$ the almost sure convergence of $N_t^{(n)}/n^{e^{-t}}$ as $n\to\infty$. This almost sure convergence follows from the construction of the Bolthausen--Sznitman coalescent as clusters of path-connected vertices in a random recursive tree by removing edges at random as time evolves. In \cite{moehle15} it is shown via exact moment calculations that $(N_t^{(n)}/n^{e^{-t}})_{t\ge0}$ converges in $D_{[0,\infty)}[0,\infty)$ as $n\to\infty$. In \cite{moehlevetter21}, the authors obtain the convergence of the scaled block counting process in the Skorohod space for a more general class of $\Lambda$-coalescents, where $\Lambda$ is essentially a beta distribution with parameters $1$ and $b>0$.

We extend the results of \cite{moehlevetter21} not only to a larger class of $\Lambda$-coalescents but even to a large class of $\Xi$-coalescents. Our key assumption (\ref{eq_res_kappa}) covers the class of $\Lambda$-coalescents treated in \cite{moehlevetter21}, as shown in Section \ref{sec_ex}.
The coalescents treated in this paper stay infinite, most coalescents with dust are included but many dust-free coalescents are covered as well. The key assumption (\ref{eq_res_kappa}) involves a certain rate function $\gamma$ known from the literature, which roughly speaking describes the expected size of a jump of the block counting process. The main result (Theorem \ref{thm_main}) states that, for a properly chosen scaling $v(n,t)$, the process $(\log N_t^{(n)}-\log v(n,t))_{t\ge0}$ converges in $D_{\rz}[0,\infty)$ as $n\to\infty$ to an Ornstein--Uhlenbeck type process. For information on Ornstein--Uhlenbeck type processes we refer the reader exemplary to \cite{satoyamazato84}.

The work of Limic \cite{limic10} is concerned with the small-time behavior of the block counting process $(N_t)_{t\ge0}:=(|\Pi_t|)_{t\ge 0}$ of $\Xi$-coalescents $(\Pi_t)_{t\ge0}$ that come down from infinity. See also \cite{berestycki10} and \cite{limic15} for $\Lambda$-coalescents. Under the regularity condition (cf. \cite[Eq. (R)]{limic10})
\begin{equation}
   \int_\Delta|u|^2\,\nu({\rm d}u)\ <\ \infty,
   \label{eq_regularity}
\end{equation}
a speed $v(t)$ of coming down of infinity is defined for which $N_t/v(t)$ converges almost surely as $t\to 0+$. The scaling $v(n,t)$ in our main convergence result (Theorem \ref{thm_main}) is defined similarly to the speed $v(t)$.

The fixation line $(L_t)_{t\ge0}$ has been introduced for $\Lambda$-coalescents by H\'enard \cite{henard15} and further studied in \cite{gaisermoehle16} for $\Xi$-coalescents. It can be characterized as the Siegmund dual \cite{siegmund76} of the block counting process satisfying (\cite[Theorem 2.9]{gaisermoehle16})
\begin{equation}
	\pr(L_t^{(m)}\ge n)\ =\ \pr(N_t^{(n)}\le m),\qquad m,n\in\nz,t\ge0,
	\label{eq_intro_duality}
\end{equation}
where the upper indices denote the initial states $L_0^{(m)}=m$ and $N_0^{(n)}=n$, respectively. Theorem \ref{thm_fix_line} states the convergence of the fixation line in the Skorohod space after suitable scaling.

The paper is organized as follows. The results are presented in Section \ref{sec_res}. In Subsection \ref{sec_gamma} the function $\gamma$ and the key assumption (\ref{eq_res_kappa}) are treated. The scaling $v(n,t)$ is defined in Subsection \ref{sec_scaling} and certain properties of the scaling are collected. In Subsection \ref{sec_block_counting} the block counting process is revisited and the main convergence result is stated. Subsection \ref{sec_fix_line} provides the analogous convergence result for the fixation line. Subsection \ref{sec_summary} summarizes the obtained convergence and duality results in non-logarithmic form. Several illustrating examples are provided in Section \ref{sec_ex}, including an example which clarifies the relation to the results in \cite{moehlevetter21} for a class of $\Lambda$-coalescents and including examples of $\Xi$-coalescents with discrete
measure $\Xi$. The proofs are provided in Section \ref{sec_proofs} in the order of appearance of the respective results. The approach to prove the main convergence result is to show the uniform convergence of the associated infinitesimal generators.


\subsection{Results}\label{sec_res}

\subsubsection{The rate function \texorpdfstring{$\gamma$}{gamma}}\label{sec_gamma}

The following function $\gamma$ has been proven to be of great significance to the study of coalescents, see \cite{herrigermoehle} and, although in different form, \cite{limic10} for $\Xi$-coalescents, and \cite{berestycki10, diehlkersting19, diehlkersting19_2, limic15} for $\Lambda$-coalescents. Define $\gamma:[0,\infty)\to\rz$ via
\begin{equation}
   \gamma(x)\ :=\ a\binom{x}{2}\ +\ \int_\Delta\sum_{i\ge 1} \big((1-u_i)^x-1+xu_i\big)\nu({\rm d}u),\qquad x\ge 0.
\label{eq_gamma}
\end{equation}
The main reason why the function $\gamma$ is so important to the study of exchangeable coalescent processes is the fact that, if the coalescent is in a state with $k\in\nz$ blocks, then (see the forthcoming equation (\ref{eq_rate})) $\gamma(k)$ is the expected rate of decrease of the number of blocks.
The properties of $\gamma$ collected in the following lemma are essentially known from the (above cited) literature.
\begin{lemma}
	Let $\gamma$ be defined by (\ref{eq_gamma}). Then $\gamma(0)=\gamma(1)=0$. Moreover, $\gamma(x)>0$ for $x>1$, $\gamma(x)\le x(x-1)(a/2+\Xi_0(\Delta))$ for $x\ge 2$, and $\gamma\in C_{\infty}((0,\infty))$ with derivative
	\[ \gamma'(x)\ =\ a\bigg(x-\frac{1}{2}\bigg)\ +\ \int_\Delta\sum_{i\ge 1}\big((1-u_i)^x\log(1-u_i)+u_i\big)\nu({\rm d}u),\qquad x>0,
	\]
	and higher derivatives
	\[ \gamma^{(k)}(x)\ =\ a\delta_{k2}\ +\ \int_\Delta\sum_{i\ge 1}(1-u_i)^x\big(\log(1-u_i)\big)^k\nu({\rm d}u),\qquad x>0,k\in\nz\setminus\{1\},
	\]
	where $\delta_{kl}$ denotes the Kronecker symbol. The map $x\mapsto\gamma(x)/x$ is strictly increasing on $[1,\infty)$. In particular, the map $\gamma$ is strictly increasing on $[1,\infty)$.
\label{lem_gamma}
\end{lemma}

We now introduce a parameter which will turn out to be of fundamental interest for our purposes. Define
\begin{equation}
    \kappa\ :=\ \lim_{x\to\infty}x\gamma''(x)\ \in\ [0,\infty]
    \label{eq_kappa}
\end{equation}
whenever this limit exists in $[0,\infty]$. In this case we call $\kappa$ the \emph{asymptotic curvature} of $\gamma$ or simply the \emph{curvature parameter} of the underlying $\Xi$-coalescent. Proposition \ref{prop_gamma_equi} shows that (\ref{eq_kappa}) is intrinsically related to the behavior of the measure $\Xi$ near $0\in\Delta$. In Section \ref{sec_ex} the curvature parameter $\kappa$ is computed for several examples.

Let us briefly comment on the coming down from infinity (cdi) property of the
coalescent. Some important coalescents, for example all beta coalescents (see Example \ref{ex_beta}) and all ${\rm NLG}$-coalescents (see Example \ref{ex_nlg}), come down from infinity if and only if $\kappa=\infty$. Note however that, in general, neither $\kappa=\infty$ implies cdi (see Example \ref{ex_five}) nor cdi implies $\kappa=\infty$ (see Example 6.1 b) of \cite{herrigermoehle}).

Lemma \ref{lem_gamma} implies that, up to multiplicative constants, $\gamma(x)$ lies for all sufficiently large $x$ in between $x$ and $x(x-1)$. The key assumption (\ref{eq_res_kappa}) of our convergence theorem (Theorem \ref{thm_main}) is a more precise condition for the growth of $\gamma(x)$, see (\ref{eq_res_L}), and can be compactly stated in terms of the curvature of $\gamma$ as follows.
\begin{equation}
    \mbox{The limit $\kappa$ in (\ref{eq_kappa}) exists and is finite.}
	\label{eq_res_kappa}
\end{equation}
Using Lemma \ref{lem_gamma} it is easily seen that (\ref{eq_res_kappa}) implies that $a:=\Xi(\{0\})=0$. In particular, (\ref{eq_res_kappa}) excludes the Kingman coalescent. We will see in Section \ref{sec_block_counting} that the assumptions of Theorem \ref{thm_main} exclude all coalescents that come down from infinity and only covers coalescents that stay infinite. If Assumption A of \cite{moehlevetter21} holds with $\kappa:=b$, then (\ref{eq_res_kappa}) holds, showing that all
convergence results of \cite{moehlevetter21} are covered by the following convergence theorems.

The following Proposition \ref{prop_gamma_equivalences} provides several conditions, each being equivalent to the key assumption (\ref{eq_res_kappa}). The proof shows that Proposition \ref{prop_gamma_equivalences} holds for any function $\gamma\in C_2((0,\infty))$ such that $\gamma''$ is nonnegative and ultimately nonincreasing.
\begin{proposition}
    The following five conditions are equivalent.
	\begin{enumerate}
		\item[(i)] Assumption (\ref{eq_res_kappa}) holds, i.e., the limit $\kappa:=\lim_{x\to\infty}x\gamma''(x)$ exists and is finite.
		\item[(ii)] $\lim_{x\to\infty}(\gamma'(x)-\gamma(x)/x)=\kappa$.
		\item[(iii)] There exists a function $L:(0,\infty)\to(0,\infty)$ being slowly varying at $\infty$ such that
		\begin{equation}
			\frac{\gamma(x)}{x}\ =\ \kappa\log x\ +\ \log L(x),\qquad x>0.
		\label{eq_res_L}
		\end{equation}
		\item[(iv)] For all $y>0$ the limit $d(y):=\lim_{x\to\infty}(\gamma(yx)/(yx)-\gamma(x)/x)$ exists and $d(y)=\kappa\log y$.
		\item[(v)] $\lim_{x\to\infty}(\gamma'(yx)-\gamma'(x))=\kappa\log y$ for all $y>0$.
	\end{enumerate}
\label{prop_gamma_equivalences}
\end{proposition}
\begin{remarks}
	\begin{enumerate}
		\item Assume that the coalescent has dust. Equivalently, $\lim_{x\to\infty}\gamma(x)/x=\int_\Delta|u|\nu({\rm d}u)=:\mu<\infty$. Thus, (\ref{eq_res_L}) and, hence, all conditions of Proposition \ref{prop_gamma_equivalences} hold with $\kappa=0$ and a slowly varying function $L$ satisfying $\lim_{x\to\infty}L(x)=e^\mu<\infty$.
		Note however that there exist dust-free coalescents (even $\Lambda$-coalescents) which satisfy $\kappa:=\lim_{x\to\infty}x\gamma''(x)=0$. We refer the reader to Examples \ref{ex_nlg} and \ref{ex_particular} in Section \ref{sec_ex}.
		\item The characterization theorem for regularly varying functions \cite[Theorem 1.4.1]{binghamgoldieteugels} implies that the limit $d(y)$ in Proposition \ref{prop_gamma_equivalences} (iv) is necessarily of the form $d(y)=\kappa\log y$, $y>0$, for some $\kappa\in\rz$, if it exists, and due to $\lim_{x\to\infty}\gamma(x)/x=\int_{\Delta}|u|\nu({\rm d}u)\in[0,\infty]$, only $\kappa\ge0$ can occur.
		\item In the terminology of \cite[Section 3]{binghamgoldieteugels}, the function $\gamma'$ is a de Haan function with $1$-index $\kappa$.
	\end{enumerate}
\end{remarks}
Proposition \ref{prop_gamma_equivalences} provides conditions being equivalent to
the key assumption (\ref{eq_res_kappa}). However, all these conditions involve the rate function $\gamma$. Proposition \ref{prop_gamma_equi} below provides two additional equivalent conditions of assumption (\ref{eq_res_kappa}), 
which do not involve the rate function $\gamma$ anymore and are instead more directly stated in terms of the measure $\Xi$ of the coalescent and hence more intuitive to understand. Proposition \ref{prop_gamma_equi} essentially shows how (\ref{eq_res_kappa}) is related to the behavior of the measure $\Xi$ near the point $0\in\Delta$. In order to state the result, let us introduce the functions $F,F_1,F_2,\ldots:[0,1)\to[0,\infty)$ and
   $G,G_1,G_2,\ldots:[0,1]\to[0,\infty)$ via
   \[
   \begin{array}{lcl}
   F_i(t) & := & \displaystyle\int_\Delta 1_{[0,t]}(u_i)(\log(1-u_i))^2\,\nu({\rm d}u),\qquad i\in\nz,t\in[0,1),\\
   F(t) &:= & \displaystyle\sum_{i\ge 1} F_i(t)
   \ =\ \int_\Delta \sum_{i\ge 1}1_{[0,t]}(u_i)(\log(1-u_i))^2\,\nu({\rm d}u),
   \qquad t\in[0,1),\\
   G_i(t) & := & \displaystyle\int_\Delta 1_{[0,t]}(u_i)u_i^2\,\nu({\rm d}u),\qquad i\in\nz,t\in [0,1],\\
   G(t) & := & \displaystyle\sum_{i\ge 1} G_i(t)
   \ =\ \int_\Delta \sum_{i\ge 1}1_{[0,t]}(u_i)u_i^2\,\nu({\rm d}u),\qquad t\in[0,1].
   \end{array}
   \]
   Note that $F_i(0)=G_i(0)=0$ for all $i\in\nz$ and, hence, $F(0)=G(0)=0$. For
   every $t\in(0,1)$ there exists a constant $C_t\in(0,\infty)$ (choose, for example, $C_t:=(-\log(1-t))/t$) such that $-\log(1-x)\le C_tx$ for all $x\in[0,t]$. Applying this inequality with $x:=u_i\le t$ yields
   \begin{eqnarray*}
      F_i(t)
      & \le & F(t)
      \ =\ \int_\Delta\sum_{i\ge 1}1_{[0,t]}(u_i)(-\log(1-u_i))^2\,\nu({\rm d}u)\\
      & \le & C_t^2\int_\Delta \sum_{i\ge 1} 1_{[0,t]}(u_i)u_i^2\,\nu({\rm d}u)
      \ \le\ C_t^2\int_\Delta\sum_{i\ge 1}u_i^2\,\nu({\rm d}u)
      \ =\ C_t^2\Xi(\Delta)\ <\ \infty.
   \end{eqnarray*}
   Obviously, $G_i(t)\le G(t)\le\int_\Delta\sum_{i\ge 1}u_i^2\nu({\rm d}u)=\Xi(\Delta)<\infty$.
   From $u_i\le -\log(1-u_i)$ we conclude that $G_i(t)\le F_i(t)$ for all $i\in\nz$ and $t\in[0,1)$ and, hence, $G(t)\le F(t)$ for all $t\in[0,1)$. Moreover, the functions $F,G,F_1,G_1,F_2,G_2,\ldots$ are nondecreasing, hence Riemann integrable.
\begin{proposition} \label{prop_gamma_equi}
   Let $\Xi$ be a finite measure on $\Delta$ and let $\kappa$ be some constant in $[0,\infty)$. Then the following three conditions are equivalent.
   \begin{enumerate}
      \item[(i)] Assumption (\ref{eq_res_kappa}) holds, i.e., the limit $\kappa=\lim_{x\to\infty}x\gamma''(x)$ exists and is finite.
      \item[(ii)] $\lim_{t\to 0+}t^{-1}F(t)=\kappa$.\qquad
            (iii) $\lim_{t\to 0+}t^{-1}G(t)=\kappa$.
   \end{enumerate}
   In particular, for $\Lambda$-coalescents, (\ref{eq_res_kappa}) is equivalent to
   \begin{equation} \label{lambda}
      \lim_{t\to0+}\frac{\Lambda([0,t])}{t}\ =\ \kappa.
   \end{equation}
\end{proposition}
Note that, if $0<\kappa<\infty$, then Proposition \ref{prop_gamma_equi} shows that, for $\Lambda$-coalescents, (\ref{eq_res_kappa}) is equivalent to the property that the measure defining function $t\mapsto\Lambda([0,t])$ of the measure $\Lambda$ is regularly varying at $0$ with index $1$. 
Relation (\ref{lambda}) already appears in Lemma 9.1 of \cite{moehlevetter21}, but its importance was not (fully) discovered there.

\subsubsection{The scaling function}\label{sec_scaling}

Define $v:[1,\infty)\times[0,\infty)\to[1,\infty)$ (implicitly) via
\begin{equation}
    v(1,t)\ :=\ 1\quad\mbox{and}\quad
	\int_{v(x,t)}^{x}\frac{{\rm d}u}{\gamma(u)}\ =\ t,\qquad x>1,t\ge 0.
\label{eq_integral_normalizing}
\end{equation}
The following two propositions clarify the existence of $v$ and provide basic properties of $v$ with an emphasis on coalescents with dust, coalescents that come down from infinity and coalescents that satisfy the key assumption (\ref{eq_res_kappa}).
\begin{proposition}
    For each $x>1$ and $t\ge 0$ the solution $v(x,t)\in(1,x]$ to the integral equation in (\ref{eq_integral_normalizing}) exists and is unique. Moreover, $v\in C_1((1,\infty)\times[0,\infty))$ with
	\begin{equation}
		\frac{{\rm d}}{{\rm d}t}v(x,t)\ =\ -\gamma(v(x,t)),\quad\frac{{\rm d}}{{\rm d}x}v(x,t)\ =\ \frac{\gamma(v(x,t))}{\gamma(x)},\qquad x>1,t\ge0.
	\label{eq_res_normalizing_derivative}
	\end{equation}
	For every $x\ge 1$ the map $t\mapsto v(x,t)$, $t\ge 0$, is nonincreasing and for every $t\ge 0$ the map $x\mapsto v(x,t)$, $x\ge 1$, is nondecreasing.
\label{prop_res_normalizing}
\end{proposition}

\begin{remark} 
	If the coalescent is in a state with $k\in\nz$ blocks, then $\gamma(k)$ is the expected rate of decrease of the block counting process. The choice of the scaling $v(x,t)$ then becomes plausible as, for each $x\ge 1$, it is the solution to the initial value problem
\begin{equation}
	\frac{{\rm d}}{{\rm d}t}v(x,t)\ =\ -\gamma(v(x,t)),\qquad t\ge 0,\qquad v(x,0)\ =\ x.
\end{equation}
\end{remark}
\begin{proposition}
	Let $\gamma$ be defined by (\ref{eq_gamma}) and let $v$ be defined by (\ref{eq_integral_normalizing}).
	\begin{enumerate}
		\item[(i)] If the coalescent has dust, i.e., $a:=\Xi(\{0\})=0$ and $\mu:=\int_\Delta|u|\nu({\rm d}u)<\infty$, then $v(x,t)\sim xe^{-\mu t}$ as $x\to\infty$ for every $t\ge0$.
		\item[(ii)] Suppose that $\int_{c}^{\infty}(\gamma(u))^{-1}{\rm d}u<\infty$ for some (and hence all) $c>1$. Then, for every $t>0$, the solution $v(t)\in(1,\infty)$ to the equation
		\begin{equation}
			\int_{v(t)}^{\infty}\frac{{\rm d}u}{\gamma(u)}\ =\ t
		\label{eq_res_speed_cdi}
		\end{equation}
		exists and $\lim_{x\to\infty}v(x,t)=v(t)$.
		\item[(iii)] Suppose that (\ref{eq_res_kappa}) holds. Then, for every $t\ge 0$, there exists a slowly varying function $L_t:[1,\infty)\to(0,\infty)$ such that $v(x,t)=x^{e^{-\kappa t}}L_t(x)$ for all $x\ge1$.
	\end{enumerate}
	\label{prop_normalizing}
\end{proposition}
\begin{remarks}
	\begin{enumerate}
		\item If the coalescent has dust, then, as $n\to\infty$, $N^{(n)}/n$ converges in $D_{[0,1]}[0,\infty)$ to the so-called frequency of singletons process \cite{gaisermoehle16}, so $(v(n,t))_{n\in\nz}$ as in (i) is a reasonable scaling sequence for the block counting process.
		\item For regular $\Xi$-coalescents that come down from infinity the integral $\int_c^{\infty}(\gamma(u))^{-1}{\rm d}u$ is finite for all $c>1$. The function $v(t)$ defined by (\ref{eq_res_speed_cdi}) is the `speed of coming down from infinity' as defined in \cite{limic10}, although with a slightly different function $\gamma$. See also \cite{berestycki10} and \cite{limic15}.
   \item The finiteness of the integral $\int_c^\infty(\gamma(u))^{-1}{\rm d}u$ can be viewed as a Grey's condition for the $\Xi$-coalescent. Grey's condition originally stems
       (see \cite{grey74, silverstein68, silverstein69}) from the study of continuous-state branching processes with the function $\gamma$ replaced by the branching mechanism of the considered branching process. For $\Lambda$-coalescents, the rate function $\gamma$ itself is the branching mechanism of a continuous-state branching process.
\end{enumerate}
\end{remarks}

As seen in Lemma \ref{lem_gamma}, the asymptotic growth as $x\to\infty$ of $\gamma(x)$ is at least of order $x$. Part (i) of the following proposition shows that, altering $\gamma$ additively by a function of asymptotic order smaller than $x$, asymptotically does essentially not change the scaling function $v(.,t)$. 
In part (ii) the slowly varying function $L_t$ of the scaling function is asymptotically calculated for a special case. 
\begin{proposition} Let $\gamma$ and $v(x,t)$ be defined by (\ref{eq_gamma}) and (\ref{eq_integral_normalizing}), respectively, and suppose that (\ref{eq_res_kappa}) holds.
	\begin{enumerate}
		\item[(i)] Assume that there exists a continuous function $\gamma_1:(1,\infty)\to(0,\infty)$ such that $(\gamma(x)-\gamma_1(x))/x\to0$ as $x\to\infty$. Then, for each $t\ge0$, there exists $x_0(t)>1$ such that the scaling $v_1(x,t)$, defined by the integral equation in (\ref{eq_integral_normalizing}) with $\gamma_1$ in place of $\gamma$, exists for all $x\ge x_0(t)$. Moreover, $v(x,t)\sim v_1(x,t)$ as $x\to\infty$.
		\item[(ii)] According to Proposition \ref{prop_gamma_equivalences}, $\gamma$ satisfies (\ref{eq_res_L}) with $\kappa\in[0,\infty)$ and a slowly varying function $L:(0,\infty)\to(0,\infty)$. If $L(x)\to C$ as $x\to\infty$ for some constant $C>0$, then, for each $t\ge 0$, $v(x,t)\sim x^{e^{-\kappa t}}C^{-\kappa^{-1}(1-e^{-\kappa t})}$ as $x\to\infty$ if $\kappa>0$ and $v(x,t)\sim xC^{-t}$ as $x\to\infty$ if $\kappa=0$.
	\end{enumerate}
	\label{prop_normalizing_constant_L}
\end{proposition}
\begin{remark}
	Assume that the coalescent has dust or, equivalently, that $\mu:=\lim_{x\to\infty}\gamma(x)/x<\infty$. Then $\gamma$ satisfies (\ref{eq_res_L}) with $\kappa=0$. Thus, $L(x)=e^{\gamma(x)/x}$ and, hence, $\lim_{x\to\infty}L(x)=e^\mu\in[1,\infty)$. 
	By Proposition \ref{prop_normalizing_constant_L} (ii), $v(x,t)\sim xe^{-\mu t}$ as $x\to\infty$ for all $t\ge0$, which also proves Proposition \ref{prop_normalizing} (i). For dust-free coalescents, $\lim_{x\to\infty}v(x,t)/x=0$ for all $t>0$.
\end{remark}

\subsubsection{Results concerning the block counting process} \label{sec_block_counting}

Let $n\in\nz$. The block counting process $(N_t^{(n)})_{t\ge0}$ with initial state $N_0^{(n)}=n$ jumps from state $k\in\{2,\ldots,n\}$ to state $j\in\{1,\ldots,k-1\}$ at the rate (see \cite[Eq.~(1.3)]{freundmoehle09} or \cite[Proposition 2.1]{gaisermoehle16})
\[
	q_{k,j}\ =\ a\binom{k}{2}1_{\{j=k-1\}}\ +\ \int_\Delta\sum_{i=1}^{j}f_{kji}(u)\nu({\rm d}u),
\]
where
\[
	f_{kji}(u)\ :=\ \sum_{\substack{k_1,\ldots,k_i\in\nz\\k_1+\cdots+k_i=k-j+i}}\frac{k!}{(j-i)!k_1!\cdots k_i!}(1-|u|)^{j-i}\sum_{\substack{l_1,\ldots,l_i\in\nz\\l_1<\cdots<l_i}}u_{l_1}^{k_1}\cdots u_{l_i}^{k_i}
\]
for $i\in\{1,\ldots,j\}$ and $u\in\Delta$.

From the Poisson point process construction of the coalescent it follows that the jump rates of the block counting process can be described in terms of an urn model as in \cite{moehle10} as follows. Fix $u\in\Delta$ and partition the interval $[0,1)$ into `urns' $J_0,J_1,\ldots$ of lengths $u_0:=1-|u|,u_1,u_2,\ldots$, i.e., $J_0:=[0,u_0),J_1:=[u_0,u_0+u_1), J_2:=[u_0+u_1,u_0+u_1+u_2)$ and so on. The `balls' $Z_1,Z_2,\ldots$ are i.i.d.~random variables, where $Z_1$ has an uniform distribution on $[0,1)$. Let $X_i(k,u):=\sum_{j=1}^k1_{\{Z_j=i\}}$ denote the number of balls in urn $i\in\nz_0:=\{0,1,\ldots\}$ after $k\in\nz_0$ throws. Note that
\begin{equation}
    Y(k,u)\ :=\ X_0(k,u)\ +\ \sum_{i\ge 1}1_{\{X_i(k,u)>0\}},\qquad k\in\nz,u\in\Delta,
    \label{eq_y}
\end{equation}
is the sum of the number of balls in urn $0$ and the number of all other occupied urns. Then
\begin{equation}
	q_{k,j}\ =\ a\binom{k}{2}1_{\{j=k-1\}}\ +\ \int_\Delta\pr(Y(k,u)=j)\nu({\rm d}u),\qquad j,k\in\nz,j<k.
\label{eq_rates_block_counting}
\end{equation}
This representation of the jump rates will turn out to be crucial to the proof of the main convergence result (Theorem \ref{thm_main}). Relation (\ref{eq_rates_block_counting}) also provides further insight into the function $\gamma$. For example, by (\ref{eq_rates_block_counting}), for all $k\in\nz$,
\begin{eqnarray}
      \sum_{j=1}^{k-1}(k-j)q_{k,j}
      & = &a\binom{k}{2}\ +\ \int_\Delta\me(k-Y(k,u))\nu({\rm d}u)\nonumber\\
      & = &a\binom{k}{2}\ +\ \int_\Delta\big(k-k(1-|u|)-\sum_{i\ge 1} (1-(1-u_i)^k)\big)\nu({\rm d}u)
      \ =\ \gamma(k). \label{eq_rate}
\end{eqnarray}
Thus, if the block counting process is in state $k\in\nz$, then $\gamma(k)=\sum_{j=1}^{k-1}(k-j)q_{k,j}$ is the expected rate of decrease of the block counting process.
Lemma \ref{lem_conv_Y} provided in the appendix shows that $Y(k,u)/k\to 1-|u|=:u_0$ almost surely as $k\to\infty$ for every $u\in\Delta$.

Define $\Delta^*:=\{u\in\Delta:|u|=1\}$. Assume that $\nu(\Delta^*)=0$ and that the regularity condition (\ref{eq_regularity}) holds. Define the function $\psi:\rz\to\cz$ via
\begin{equation}
	\psi(x)\ :=\ \int_\Delta\big((1-|u|)^{ix}-1+ix|u|\big)\nu({\rm d}u),\qquad x\in\rz.
\label{eq_psi}
\end{equation}
Note that (\ref{eq_regularity}) ensures that $\psi(x)\in\cz$. Define the transformation $g:\Delta\setminus\Delta^*\to(-\infty,0)$ via $g(u):=\log(1-|u|)$ for all $u\in\Delta\setminus\Delta^*$ and let $\varrho:=\nu_g$ denote the image measure of $\nu$ under $g$. Then,
\begin{eqnarray*}
	\psi(x)& = &\int_{(-\infty,0)}\big(e^{ixt}-1+ix(1-e^t)\big)\varrho({\rm d}t)\\& = & ix\int_{(-\infty,0)}\bigg(1-e^t+\frac{t}{1+t^2}\bigg)\varrho({\rm d}t)\ +\ \int_{(-\infty,0)}\bigg(e^{ixt}-1-ix\frac{t}{1+t^2}\bigg)\varrho({\rm d}t),~~x\in\rz.
\end{eqnarray*}
Thus, $\psi$ is the characteristic exponent of an infinitely divisible distribution. The regularity condition (\ref{eq_regularity}) is required for $\varrho:=\nu_g$ to be a L\'evy measure. In the following, for $t\ge0$, $S_t$ denotes a random variable with characteristic function $\phi_t$, given by
\begin{equation}
	\phi_t(x)\ :=\ \exp\bigg(\int_0^t\psi(e^{-\kappa s}x){\rm d}s\bigg),\qquad x\in\rz,t\ge 0,
	\label{eq_res_char_funct_S_t}
\end{equation}
where $\psi$ is defined by (\ref{eq_psi}).

The limiting process $X$ arising in the main convergence result (Theorem \ref{thm_main} below), whose distribution is determined through its semigroup $(T_t^X)_{t\ge0}$ by (\ref{eq_res_char_funct_S_t}) and (\ref{eq_semigroup_lim}), belongs to the class of Ornstein--Uhlenbeck type processes \cite{satoyamazato84}. The semigroup $(T_t^X)_{t\ge0}$ belongs to the class of generalized Mehler semigroups (see \cite{bogachevroecknerschmuland96}), since $\phi_{t+s}(x)=\phi_t(e^{-\kappa s}x)\phi_s(x)$ for $x\in\rz$ and $s,t\ge0$. Clearly, $(T_t^X)_{t\ge 0}$ is a Feller semigroup 
on $\widehat{C}(\rz)$, the space of continuous functions $f:\rz\to\rz$ vanishing at infinity.

\begin{theorem}
	Suppose that $\Xi$ satisfies (\ref{eq_regularity}) and $\Xi(\Delta^*)=0$. Let $\gamma$ be defined by (\ref{eq_gamma}) and suppose that (\ref{eq_res_kappa}) holds, i.e., the limit $\kappa:=\lim_{x\to\infty}x\gamma''(x)\in[0,\infty)$ exists. Moreover, let the scaling $v(n,t)$ be defined by (\ref{eq_integral_normalizing}). Then the logarithmically scaled block counting process $(\log N_t^{(n)}-\log v(n,t))_{t\ge0}$ converges in $D_{\rz}[0,\infty)$ to $X$ as $n\to\infty$, where $X=(X_t)_{t\ge0}$ is an Ornstein--Uhlenbeck type process with state space $\rz$, initial value $X_0=0$ and Mehler semigroup $(T_t^X)_{t\ge0}$ given by
	\begin{equation}
	T_t^Xf(x)\ :=\ \me(f(X_{s+t})|X_s=x)\ =\ \me(f(e^{-\kappa t}x+S_t)),\qquad x\in\rz,f\in B(\rz),s,t\ge 0,
	\label{eq_semigroup_lim}
	\end{equation}
    with the distribution of $S_t$ defined via its characteristic function (\ref{eq_res_char_funct_S_t}).
	\label{thm_main}
\end{theorem}
\begin{remark}
    For results on the generator $A^X$ of the limiting process $X$ arising in Theorem \ref{thm_main} we refer the reader to (\ref{eq_lim_process_generator_0}).
\end{remark}
\begin{remark}
	If (\ref{eq_res_kappa}) and, hence, (\ref{eq_res_L}) holds, then $\int_c^{\infty}(\gamma(u))^{-1}{\rm d}u=\infty$ for some (and hence all) $c\in(1,\infty)$. For $\varepsilon\in(0,1)$ define $\Delta^\epsilon:=\{u\in\Delta:|u|\le1-\varepsilon\}$. Further, define $\Delta_{f}:=\{u\in\Delta:u_1+\ldots+u_n=1~\text{for some $n\in\nz$}\}$. Under the regularity condition (\ref{eq_regularity}) it holds that $\nu(\Delta\setminus\Delta^\varepsilon)<\infty$ for all $\varepsilon\in(0,1)$ (see \cite[p. 229]{limic10}). Due to $\Xi(\Delta_{f})=0$, the coalescents covered by Theorem \ref{thm_main} hence stay infinite \cite[Proposition 33]{schweinsberg00}. In other words the coalescents covered by Theorem \ref{thm_main} either have dust or they have no dust and are not coming down from infinity. A schematic representation of the space ${\cal M}(\Delta)$ of all finite measures $\Xi$ on $(\Delta,{\cal B}(\Delta))$ is provided in Figure \ref{figure_1}. In this representation, ${\cal M}(\Delta)$ is equipped with the topology of weak convergence, i.e., $\Xi_n\to\Xi$ as $n\to\infty$ if and only if $\lim_{n\to\infty}\int_\Delta f\,{\rm d}\Xi_n=\int_\Delta f\,{\rm d}\Xi$ for all continuous functions $f:\Delta\to\rz$. Note that, since $\Delta$ is compact, all continuous functions $f:\Delta\to\rz$ are bounded and uniformly continuous. The space ${\cal M}(\Delta)$ is metrizable, for example via the metric
\[
d(\Xi_1,\Xi_2)
\ :=\ \sum_{i\ge 1} 2^{-i} \frac{1}{1+\|f_i\|}
\bigg|\int f_i\,{\rm d}\Xi_1-\int f_i\,{\rm d}\Xi_2\bigg|,
\]
where $\{f_1,f_2,\ldots\}$ is a dense set of real valued continuous functions on $\Delta$. The results in Parthasarathy \cite[Chapter 6]{parthasarathy} imply that, with this metric, ${\cal M}(\Delta)$ is a compact Polish (separable
complete metric) space.
\end{remark}
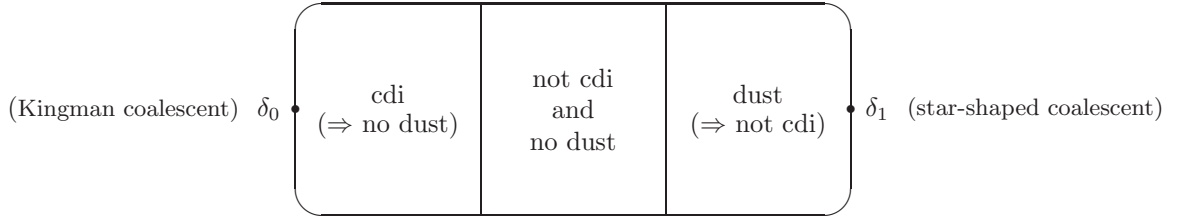
\begin{figure}[ht]
	\begin{center}
		\begin{picture}(210,80)
		\put(10,0){\line(1,0){190}}
		\put(10,80){\line(1,0){190}}
		\put(0,10){\line(0,1){60}}
		\put(70,0){\line(0,1){80}}
		\put(140,0){\line(0,1){80}}
		\put(210,10){\line(0,1){60}}
		\put(10,10){\oval(20,20)[bl]}
		\put(10,70){\oval(20,20)[tl]}
		\put(200,10){\oval(20,20)[br]}
		\put(200,70){\oval(20,20)[tr]}
		\put(0,40){\circle*{3}}
		\put(-10,40){\makebox(0,0)[]{$\delta_0$}}
		\put(-65,40){\makebox(0,0)[]{(\small Kingman coalescent)}}
		\put(210,40){\circle*{3}}
		\put(220,40){\makebox(0,0)[]{$\delta_1$}}
		\put(280,40){\makebox(0,0)[]{\small (star-shaped coalescent)}}
		\put(35,46){\makebox(0,0)[]{cdi}}
		\put(35,34){\makebox(0,0)[]{($\Rightarrow$ no dust)}}
		\put(175,46){\makebox(0,0)[]{dust}}
		\put(175,34){\makebox(0,0){($\Rightarrow$ not cdi)}}
		\put(105,52){\makebox(0,0)[]{not cdi}}
		\put(105,40){\makebox(0,0)[]{and}}
		\put(105,28){\makebox(0,0)[]{no dust}}
		\end{picture}
	\end{center}
	\caption{A schematic representation of the space of all exchangeable coalescents ($\Xi$-coalescents). Each point in the oval region corresponds to a finite measure $\Xi$ on $(\Delta,{\cal B}(\Delta))$. The compact Polish space ${\cal M}(\Delta)$ is divided into three regions of exchangeable coalescents, those coming down from infinity (cdi) to the left, the ones not coming down from infinity and having no dust in the middle and those having dust to the right.}
\label{figure_1}
\end{figure}

\subsubsection{Results concerning the fixation line}\label{sec_fix_line}
The fixation line has been introduced by H\'enard \cite{henard15} for $\Lambda$-coalescents and further studied in \cite{gaisermoehle16} for general $\Xi$-coalescents. The fixation line
$(L_t^{(n)})_{t\ge 0}$ with initial state $L_0^{(n)}=n$ is a Markov process
which moves from state $i\in\{n,n+1,\ldots\}$ to state $j\in\nz$ with $j>i$ at the rate (see \cite[Proposition 2.5]{gaisermoehle16})
\[
\gamma_{ij}\ =\ a\binom{j}{2}\delta_{j,i+1}\ +\ \int_\Delta\pr\big(Y(j,u)=i,Y(j+1,u)=i+1\big)\nu({\rm d}u),
\]
where $a:=\Xi(\{0\})$ and $Y(.,u)$ is defined via (\ref{eq_y}). The fixation line does not explode if and only if the coalescents stays infinite \cite[Remark 2.11]{gaisermoehle16}. Recall that (see (\ref{eq_intro_duality})) the block counting process is Siegmund dual to the fixation line. We will see in this subsection that this duality property transfers the convergence result for the block counting process (Theorem \ref{thm_main}) into an analogous convergence result (Theorem \ref{thm_fix_line}) for the fixation line. The arguments are similar as for the fixation line. We start as follows. Assume that $\int_2^\infty(\gamma(u))^{-1}{\rm d}u=\infty$. Define the function $w:[1,\infty)\times[0,\infty)\to[1,\infty)$ via
\begin{equation}
w(1,t)\ :=\ 1\quad\mbox{and}\quad
\int_x^{w(x,t)}\frac{{\rm d}u}{\gamma(u)}\ =\ t,\qquad x>1,t\ge 0.
\label{eq_int_normalizing_fix_line}
\end{equation}

\begin{proposition}
	Assume that $\int_2^\infty(\gamma(u))^{-1}{\rm d}u=\infty$. Then, for each $x>1$ and $t\ge0$, the solution $w(x,t)\in[x,\infty)$ to the integral equation in (\ref{eq_int_normalizing_fix_line}) exists and is unique. Furthermore, $w\in C_1((1,\infty)\times[0,\infty))$ with
	\begin{equation}
		\frac{{\rm d}}{{\rm d}t}w(x,t)\ =\ \gamma(w(x,t)),\quad\frac{{\rm d}}{{\rm d}x}w(x,t)\ =\ \frac{\gamma(w(x,t))}{\gamma(x)}\qquad x>1,t\ge0.
	\label{eq_derivative_normalizing_fix_line}
	\end{equation}
	The maps $x\mapsto w(x,t)$, $x\ge1$, and $t\mapsto w(x,t)$, $t\ge 0$, are strictly increasing.
\label{prop_normalizing_fix_line}	
\end{proposition}

It is readily seen from (\ref{eq_integral_normalizing}) and (\ref{eq_int_normalizing_fix_line}) that $v(w(x,t),t)=x=w(v(x,t),t)$ for all $x\ge1$ and $t\ge0$. Thus, for fixed $t\ge 0$, $w(.,t)$ is the inverse of $v(.,t)$. This aspect is utilized in the proof of Proposition \ref{prop_normalizing_fix_line_kappa} below, whose statements are variants of Propositions \ref{prop_normalizing} and \ref{prop_normalizing_constant_L} for the scaling of the fixation line. Note that, under the key assumption (\ref{eq_res_kappa}), in particular when the coalescent has dust, it holds that $\int_{c}^{\infty}(\gamma(u))^{-1}{\rm d}u=\infty$ for every $c>1$, so $w(x,t)$ is well-defined.

\begin{proposition}
	Let $\gamma$ be defined by (\ref{eq_gamma}) and $w$ be defined by (\ref{eq_int_normalizing_fix_line}).
	\begin{enumerate}
		\item[(i)] If $a:=\Xi(\{0\})=0$ and $\mu:=\int_{\Delta}|u|\nu({\rm d}u)<\infty$, then $w(x,t)\sim xe^{\mu t}$ as $x\to\infty$ for each $t\ge0$.
		\item[(ii)] Suppose that $\gamma$ satisfies (\ref{eq_res_kappa}) with $\kappa\ge0$. Then, for every $t\ge 0$, there exists a slowly varying function $L_t^{\#}:[1,\infty)\to(0,\infty)$ such that $w(x,t)=x^{e^{\kappa t}}L_t^{\#}(x)$ for all $x\ge 1$.
		\item[(iii)] 
		Suppose that $\gamma$ satisfies (\ref{eq_res_kappa}) with $\kappa\ge0$. Assume that there exists a continuous function $\gamma_1:(1,\infty)\to(0,\infty)$ such that $(\gamma(x)-\gamma_1(x))/x\to0$ as $x\to\infty$. Then the scaling $w_1(x,t)$, defined by (\ref{eq_int_normalizing_fix_line}) with $\gamma_1$ in place of $\gamma$, exists for all $t\ge0$ and $x\ge 1$. Moreover, assume that the map $x\mapsto\gamma_1(x)/x,$ $x>1,$ is nondecreasing if $\kappa=0$. Then $w(x,t)\sim w_1(x,t)$ as $x\to\infty$.
	\end{enumerate}
\label{prop_normalizing_fix_line_kappa}
\end{proposition}
\begin{remark}
  	The regular variation of $w(.,t)$ under the key assumption (\ref{eq_res_kappa}) is a consequence of the regular variation of $v(.,t)$ (see Proposition \ref{prop_normalizing}) and the fact that $w(.,t)$ and $v(.,t)$ are inverse. The slowly varying part $L_t^{\#}$ can be retrieved from $L_t$ with the use of the de Bruijn conjugate, see \cite[Theorem 1.5.13 and Proposition 1.5.15]{binghamgoldieteugels} and the proof of Proposition \ref{prop_normalizing_fix_line_kappa} in Section \ref{sec_proof_fix_line} for further details.
\end{remark}

The following theorem is the analog of Theorem \ref{thm_main} for the fixation line.

\begin{theorem}
	Suppose that $\Xi$ satisfies (\ref{eq_regularity}) and $\Xi(\Delta^*)=0$. Let $\gamma$ be defined by (\ref{eq_gamma}) and suppose that (\ref{eq_res_kappa}) holds with $0\le\kappa<\infty$. 
	Let the scaling $w(n,t)$ be defined by (\ref{eq_int_normalizing_fix_line}). Then $(\log L_t^{(n)}-\log w(n,t))_{t\ge0}$ converges in $D_{\rz}[0,\infty)$ to $Y$ as $n\to\infty$, where $Y=(Y_t)_{t\ge0}$ is an Ornstein--Uhlenbeck type process with state space $\rz$, initial value $Y_0=0$ and Mehler semigroup $(T_t^Y)_{t\ge0}$ given by
	\begin{equation}
		T_t^Yf(y)\ :=\ \me(f(Y_{s+t})|Y_s=y)\ =\ \me(f(e^{\kappa t}y-e^{\kappa t}S_t)),\qquad y\in\rz,f\in B(\rz),s,t\ge0,
	\label{eq_semigroup_lim_fix_line}
	\end{equation}
    with the distribution of $S_t$ defined via its characteristic function (\ref{eq_res_char_funct_S_t}).
\label{thm_fix_line}
\end{theorem}



\subsubsection{Siegmund duality and summary of results} \label{sec_summary}
Let $X=(X_t)_{t\ge0}$ and $Y=(Y_t)_{t\ge0}$ be the limiting processes arising in Theorems \ref{thm_main} and \ref{thm_fix_line}, respectively. For $t\ge 0$ define $\widetilde{X}_t:=e^{X_t}$ and $\widetilde{Y}_t:=e^{Y_t}$. Consider the `exponential' Markov processes $\widetilde{X}:=(\widetilde{X}_t)_{t\ge 0}$ and $\widetilde{Y}:=(\widetilde{Y}_t)_{t\ge 0}$ both having state space $E:=(0,\infty)$. From (\ref{eq_semigroup_lim}) and (\ref{eq_semigroup_lim_fix_line}) it follows that the semigroups $(T_t^{\widetilde{X}})_{t\ge 0}$ and $(T_t^{\widetilde{Y}})_{t\ge 0}$ of $\widetilde{X}$ and $\widetilde{Y}$ are given by
\begin{equation} \label{semixtilde}
T_t^{\widetilde{X}}f(x)\ =\ \me(f(x^{e^{-\kappa t}}e^{S_t})),\qquad t\ge 0, f\in B(E), x\in E,
\end{equation}
and
\begin{equation} \label{semiytilde}
T_t^{\widetilde{Y}}g(y)\ =\ \me(g(y^{e^{\kappa t}}e^{-e^{\kappa t}S_t})),\qquad t\ge 0,g\in B(E),y\in E,
\end{equation}
where $S_t$ has characteristic function (\ref{eq_res_char_funct_S_t}).

Fix $t\ge 0$, define $\alpha:=e^{-\kappa t}$ and let $H:E\times E\to\{0,1\}$ denote the Siegmund duality kernel, i.e., $H(x,y):=1$ for $x\le y$ and $H(x,y):=0$ otherwise. For $x,y\in E$, by (\ref{semixtilde}), $T_t^{\widetilde{X}}H(.,y)(x)=\me(H(x^\alpha e^{S_t},y))=\pr(x^\alpha e^{S_t}\le y)$. Similarly, by (\ref{semiytilde}), $T_t^{\widetilde{Y}}H(x,.)(y)=\pr(H(x,y^{1/\alpha}e^{-S_t/\alpha}))
=\pr(x\le y^{1/\alpha}e^{-S_t/\alpha})=\pr(x^\alpha e^{S_t}\le y)$. Thus, $T_t^{\widetilde{X}}H(.,y)(x)=T_t^{\widetilde{Y}}H(x,.)(y)$ for all $t\ge 0$ and $x,y\in E$, showing that $\widetilde{X}$ is Siegmund dual to $\widetilde{Y}$.

Since the map $D_\rz[0,\infty)\ni x=(x_t)_{t\ge 0}\mapsto (e^{x_t})_{t\ge0}\in D_E[0,\infty)$ is continuous, an application of the continuous mapping theorem shows that Theorems \ref{thm_main} and \ref{thm_fix_line} can be summarized as follows.
\begin{theorem} \label{thm_main3}
   Under the conditions of Theorem \ref{thm_main} the following two assertions hold.
   \begin{enumerate}
      \item[i)] As $n\to\infty$, the scaled block counting process $(N_t^{(n)}/v(n,t))_{t\ge 0}$ converges in $D_E[0,\infty)$ to the Markov process $\widetilde{X}$ with $\widetilde{X}_0=1$ and semigroup (\ref{semixtilde}).
      \item[ii)] As $n\to\infty$, the scaled fixation line $(L_t^{(n)}/w(n,t))_{t\ge 0}$ converges in $D_E[0,\infty)$ to the Markov process $\widetilde{Y}$ with $\widetilde{Y}_0=1$ and semigroup (\ref{semiytilde}).
   \end{enumerate}
\end{theorem}
Thus, under the assumptions of Theorem \ref{thm_main}, the commutative diagram in Figure 2 holds.
\begin{figure}[ht]
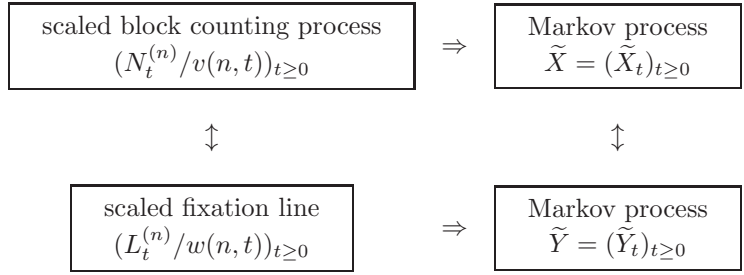

\[
\begin{array}{ccc}
\fbox{
      \begin{tabular}{c}
         scaled block counting process\\
         $(N_t^{(n)}/v(n,t))_{t\ge0}$
      \end{tabular}
}
      & \Rightarrow &
\fbox{
   \begin{tabular}{c}
      Markov process\\
      $\widetilde{X}=(\widetilde{X}_t)_{t\ge0}$
   \end{tabular}
}\\
   & & \\
   \updownarrow & & \updownarrow\\
   & & \\
\fbox{
   \begin{tabular}{c}
      scaled fixation line\\
      $(L_t^{(n)}/w(n,t))_{t\ge0}$
   \end{tabular}
}
      & \Rightarrow &
\fbox{
   \begin{tabular}{c}
      Markov process\\
      $\widetilde{Y}=(\widetilde{Y}_t)_{t\ge0}$
   \end{tabular}
}
\end{array}
\]
\caption{Commutative diagram summarizing the convergence and duality results.
In the diagram, `$\Rightarrow$' stands for convergence in $D_E[0,\infty)$ and `$\updownarrow$' for Siegmund duality $\pr(N_t^{(n)}\le m)=\pr(L_t^{(m)}\ge n)$, $n,m\in\nz$, and $\pr(\widetilde{X}_t^{(x)}\le y)=\pr(\widetilde{Y}_t^{(y)}\ge x)$, $x,y\in E$, $t\ge0$, respectively, where the upper indices indicate the initial states of the corresponding processes.}
\end{figure}

We close the result section by providing formulas for the infinitesimal generators of the processes $\widetilde{X}$ and $\widetilde{Y}$. Applying the generator formulas $A^{\widetilde{X}}f(x)=A^X(f\circ\exp)(\log x)$ and $A^{\widetilde{Y}}g(y)=A^Y(g\circ\exp)(\log y)$ yield that the generators $A^{\widetilde{X}}$ and $A^{\widetilde{Y}}$ of $\widetilde{X}$ and $\widetilde{Y}$ satisfy
\begin{equation} \label{genxtilde}
A^{\widetilde{X}}f(x)\ =\ -\kappa x(\log x)f'(x)\ +\ \int_\Delta
\big(f(x(1-|u|))-f(x)+|u|xf'(x)\big)\nu({\rm d}u)
\end{equation}
for $x>0$ and $f\in\widetilde{D}$ and
\begin{equation} \label{genytilde}
A^{\widetilde{Y}}g(y)\ =\ \kappa y(\log y)g'(y)\ +\ \int_\Delta
\big(g(y/(1-|u|))-g(y)-|u|yg'(y)\big)\nu({\rm d}u)
\end{equation}
for $y>0$ and $g\in\widetilde{D}$, where $\widetilde{D}$
denotes the space of all functions $f:E\to\rz$ such that the maps $f$,
$x\mapsto xf'(x)$, $x\mapsto x^2f''(x)$ and $x\mapsto x(\log x)f'(x)$ belong to $\widehat{C}(E)$. Note that $\widetilde{D}$ is a core for both generators, $A^{\widetilde{X}}$ and $A^{\widetilde{Y}}$.


\subsection{Examples}\label{sec_ex}

In this section several illustrating examples are provided. For most of the examples Theorem \ref{thm_main} and Theorem \ref{thm_fix_line} are applicable. Example \ref{ex_beta} treats the $\Lambda$-coalescent, where $\Lambda=\beta(a,b)$ is a beta distribution with parameters $a,b>0$. For the
$\beta(1,b)$-coalescent scaling limits have already been obtained in \cite{moehlevetter21}, and we clarify beforehand the relation between \cite{moehlevetter21} and this work. Example \ref{ex_nlg} studies the $\Lambda$-coalescent introduced in \cite{moehle21}, where the measure $\Lambda$ is a negative logarithmic gamma distribution (${\rm NLG}$-coalescent). Example \ref{ex_particular} provides a simple dust-free $\Lambda$-coalescent which nevertheless satisfies $\kappa=0$. Example \ref{ex_xi} presents a true $\Xi$-coalescent for which Theorem \ref{thm_main} and Theorem \ref{thm_fix_line} are applicable.

We start with putting the results of \cite{moehlevetter21} in the context of our work. Let $b>0$. In \cite{moehlevetter21} it is shown that $(\log N_t^{(n)}-e^{-bt}\log n)_{t\ge0}$ converges in $D_{\rz}[0,\infty)$ as $n\to\infty$ to an Ornstein--Uhlenbeck type process provided that the coalescent's driving measure $\Lambda$ satisfies
\begin{eqnarray}
	\Lambda(\{0\})\ =\ \Lambda(\{1\})\ =\ 0\quad\text{and}\quad c\ :=\ \int_{[0,1]}u^{-1}(\Lambda-b\lambda)({\rm d}u)\ <\ \infty.
\label{eq_ex_lambda}
\end{eqnarray}
Here $\lambda$ denotes Lebesgue measure. Condition
(\ref{eq_ex_lambda}) essentially forces the coalescent to behave similarly to the Bolthausen--Sznitman coalescent ({\rm BS}-coalescent), which is the $\Lambda$-coalescent where $\Lambda$ is the uniform distribution on $[0,1]$. For $x\ge 0$ define $\gamma_{{\rm BS}}(x):=\int_0^1\big((1-u)^x-1+ux\big)u^{-2}{\rm d}u$. Let $\Psi:=(\log\Gamma)'=\Gamma'/\Gamma$ denote the logarithmic derivative of the gamma function (digamma function).
It is easily checked that $\gamma_{{\rm BS}}(x)=x(\Psi(x+1)-\Psi(1)-1)=
x\log x-(\Psi(1)+1)x+O(1)$ as $x\to\infty$.
If (\ref{eq_ex_lambda}) holds, then
\[
\gamma(x)
\ =\ b\gamma_{{\rm BS}}(x)\ +\ \int_{[0,1]}\big((1-u)^x-1+xu\big)(\Lambda-b\lambda)({\rm d}u)\ =\ bx\log x+x(-b(1+\Psi(1))+c+o(1))
\]
for $x>0$ such that (\ref{eq_res_L}) and, hence, (\ref{eq_res_kappa}) are satisfied with $\kappa:=b$ and the slowly varying function $L$ in
(\ref{eq_res_L}) satisfies $L(x)\to\exp(-b(1+\Psi(1))+c)$ as $x\to\infty$.

\begin{example} \rm (beta coalescent)
   Let $\Lambda=\beta(a,b)$ be the beta distribution with parameters $a,b>0$.
   For the corresponding $\Lambda$-coalescent, the function
   $\gamma$, defined via (\ref{eq_gamma}), can be calculated explicitly.
   For $a\notin\{1,2\}$, a technical but straightforward calculation
   shows that
   \begin{equation} \label{eq_gammabeta1}
      \gamma(x)
      \ =\ \frac{(x+a+b-1)(x+a+b-2)}{(a-1)(a-2)}
      \frac{{\rm B}(a,x+b)}{{\rm B}(a,b)}
      \ +\ \frac{a+b-1}{a-1}x\ -\ \frac{(a+b-1)(a+b-2)}{(a-1)(a-2)}
   \end{equation}
   for all $x\ge 0$, where ${\rm B}(.,.)$ denotes the beta function. The boundary cases $a=1$ and $a=2$ need to be treated separately. For $a=1$ one obtains
   \begin{equation} \label{eq_gammabeta2}
      \gamma(x)\ =\ b(x+b-1)\big(\Psi(x+b)-\Psi(b)\big)-bx,\qquad x\ge 0,
   \end{equation}
   where $\Psi$ denotes the digamma function. For $a=2$ it follows that
   \begin{equation} \label{eq_gammabeta3}
      \gamma(x)\ =\ (b+1)x - b(b+1)\big(\Psi(x+b)-\Psi(b)\big),\qquad x\ge 0.
   \end{equation}
   Since $\Psi(x+b)=\log x+O(x^{-1})$ as $x\to\infty$ it
   follows from (\ref{eq_gammabeta1}), (\ref{eq_gammabeta2}) and (\ref{eq_gammabeta3})
   that
   \[
   \frac{\gamma(x)}{x}
   \left\{
   \begin{array}{cl}
      \ =\ \displaystyle\frac{\Gamma(a+b)}{\Gamma(b)(1-a)(2-a)}x^{1-a} + O(1) & \mbox{for $a<1$},\\
      \ =\ \displaystyle b\log x - b(\Psi(b)+1) + O\bigg(\frac{\log x}{x}\bigg) & \mbox{for $a=1$},\\
      \ \to\ \displaystyle\frac{a+b-1}{a-1} & \mbox{for $a>1$}.
   \end{array}
   \right.
   \]
   For all $y\in(0,\infty)$ the limit $d(y)$, defined in Proposition
   \ref{prop_gamma_equivalences} (iv), is thus given by
   \[
   d(y)\ =\ \left\{
   \begin{array}{cl}
      -\infty 1_{(0,1)}(y)+\infty 1_{(1,\infty)}(y)
         & \mbox{for $a<1$ (cdi),}\\
      b\log y & \mbox{for $a=1$ (not cdi and no dust),}\\
        0 & \mbox{for $a>1$ (dust),}
   \end{array}
   \right.
   \]
   and the curvature parameter $\kappa$ is given by
   \[
   \kappa\ :=\ \lim_{x\to\infty}x\gamma''(x)
   \ =\ \left\{
      \begin{array}{cl}
         \infty & \mbox{for $a<1$,}\\
         b & \mbox{for $a=1$,}\\
         0 & \mbox{for $a>1$.}
      \end{array}
   \right.
   \]
   For beta coalescents, the curvature parameter $\kappa$ thus characterizes both, the dust property and the cdi property. Theorems \ref{thm_main} and \ref{thm_fix_line} are hence applicable for the $\beta(a,b)$-coalescent with $a\ge 1$. Let us distinguish two cases.

   Case 1: If $a>1$ (dust case), then $\kappa=0$, $v(x,t)\sim e^{-\mu t}x$ and $w(x,t)\sim e^{\mu t}x$ as $x\to\infty$ with $\mu:=\int u^{-1}\Lambda({\rm d}u)=(a+b-1)/(a-1)$. In this case, Theorems \ref{thm_main} and \ref{thm_fix_line} are in essence logarithmic versions of \cite[Theorem 2.1]{gaisermoehle16}.

   Case 2: Assume now that $a=1$, i.e., that $\Lambda=\beta(1,b)$ is the beta distribution with parameters $1$ and $b>0$ having density $x\mapsto b(1-x)^{b-1}$, $x\in(0,1)$, with respect to Lebesgue measure. Then,
   $\kappa=b>0$.

   From the discussion above (see also \cite[Example 2 or Proposition 11]{moehlevetter21}) it follows that
	\begin{eqnarray*}
		\gamma(x)\ =\ bx\log x+x(\log C_b+o(1)),\qquad x>0,
	\end{eqnarray*}
	where $C_b:=\exp(-b(\Psi(b)+1))$. Independently one can verify that
	\begin{eqnarray*}
		x\gamma''(x)\ =\ bx\int_{0}^{1}\frac{(1-u)^{x+b-1}(\log(1-u))^2}{u^2}\,{\rm d}u\ \to\ b,\qquad x\to\infty.
	\end{eqnarray*}
	Let the scaling sequence $v(n,t)$ be defined by (\ref{eq_integral_normalizing}) for $n\ge2$. By Proposition \ref{prop_normalizing_constant_L}, $v(x,t)\sim x^{e^{-bt}}C_b^{b^{-1}(e^{-bt}-1)}=x^{e^{-bt}}e^{(\Psi(b)+1)(1-e^{-bt})}$ as $x\to\infty$. Similarly, $w(x,t)\sim x^{e^{bt}}e^{-(\Psi(b)+1)(1-e^{-bt})}$ as $x\to\infty$. By Theorems \ref{thm_main} and \ref{thm_fix_line}, both processes $(\log N_t^{(n)}-\log v(n,t))_{t\ge0}$ and $(\log L_t^{(n)}-\log w(n,t))_{t\ge 0}$ converge in $D_\rz[0,\infty)$ as $n\to\infty$. The process $(\log N_t^{(n)}-e^{-bt}\log n)_{t\ge0}$ converges as well as $n\to\infty$ due to the specifics of the scaling sequence, and the limiting processes generator $A^{X}$, which can be determined using Lemma \ref{lem_generator_core_0}, is given by
	\[
		A^{X}f(x)\ =\ bf'(x)(1+\Psi(b)-x)\ +\ \int_{[0,1]}\big(f(x+\log(1-u))-f(x)+uf'(x)\big)u^{-2}\Lambda({\rm d}u)
	\]
	for $x\in\rz$ and $f$ belonging to a core $D$, in agreement with the results of \cite{moehlevetter21}.
\label{ex_beta}
\end{example}

Further examples are now provided for which Theorems \ref{thm_main} and \ref{thm_fix_line} are applicable.

\begin{example} \rm \label{ex_nlg} (NLG-coalescent)
    Fix $\alpha,\varrho>0$. Assume that $\Lambda$ is the negative logarithmic gamma distribution having density $\alpha^\varrho u^{\alpha-1} (-\log u)^{\varrho-1}/\Gamma(\varrho)$, $u\in (0,1)$, with respect to Lebesgue measure on $(0,1)$. The corresponding $\Lambda$-coalescent was introduced in \cite[Example 3.2]{moehle21}. The asymptotics of $\gamma(x)$ as $x\to\infty$ is obtained as follows. For all $n\in\nz$, $\gamma(n)=\sum_{j=0}^{n-1}a_j$, where
    \[
    a_j\ :=\ \int\frac{1-(1-u)^j}{u}\Lambda({\rm d}u)\ \sim\ 
   \left\{
      \begin{array}{cl}
         \displaystyle\frac{\alpha^\varrho}{\Gamma(\varrho)}\frac{\Gamma(\alpha)}{1-\alpha}
         j^{1-\alpha}(\log j)^{\varrho-1}
            & \mbox{if $0<\alpha<1$,}\\
         \displaystyle\frac{(\log j)^\varrho}{\Gamma(\varrho+1)}
            & \mbox{if $\alpha=1$,}\\
         \displaystyle\bigg(\frac{\alpha}{\alpha-1}\bigg)^\varrho
            & \mbox{if $1<\alpha<\infty$,}
      \end{array}
   \right.
   \]
   as $j\to\infty$ by \cite[Lemma 7.3]{moehle21}, applied with $a:=\alpha$, $b:=j$ and $c:=\varrho$. Thus, as $n\to\infty$, the arithmetic mean $\gamma(n)/n=n^{-1}\sum_{j=0}^{n-1}a_j$ of the sequence $(a_j)_{j\in\nz_0}$
   satisfies
   \[
   \frac{\gamma(n)}{n}\ \sim\ 
   \left\{
      \begin{array}{cl}
         \displaystyle\frac{\alpha^\varrho}{\Gamma(\varrho)}
         \frac{\Gamma(\alpha)}{(1-\alpha)(2-\alpha)}
         n^{2-\alpha}(\log n)^{\varrho-1}
            & \mbox{if $0<\alpha<1$,}\\
         \displaystyle\frac{(\log n)^\varrho}{\Gamma(\varrho+1)}
            & \mbox{if $\alpha=1$,}\\
         \displaystyle\bigg(\frac{\alpha}{\alpha-1}\bigg)^\varrho
            & \mbox{if $1<\alpha<\infty$.}
      \end{array}
   \right.
   \]
   This coalescent has dust if and only if $\gamma(n)/n$ is bounded, so if and only if $\alpha>1$. This coalescent comes down from infinity if and only if $\sum_{n=2}^\infty 1/\gamma(n)<\infty$, so if and only if $\alpha<1$ or $\alpha=1$ and $\varrho>1$. It is easily seen that 
   \[
   \kappa\ :=\ \lim_{\varepsilon\to 0+}\frac{\Lambda([0,\varepsilon])}{\varepsilon}
   \ =\ \left\{
      \begin{array}{cl}
         \infty & \mbox{if $0<\alpha<1$ or if $\alpha=1$ and $1<\varrho<\infty$,}\\
         1 & \mbox{if $\alpha=\varrho=1$ (Bolthausen--Sznitman coalescent),}\\
         0 & \mbox{if $1<\alpha<\infty$ or if $\alpha=1$ and $0<\varrho<1$.}
      \end{array}
   \right.
   \]
   In particular, $\kappa=\infty$ if and only if the coalescent comes down
   from infinity. Theorems \ref{thm_main} and \ref{thm_fix_line} are applicable if and only if $\kappa<\infty$, so if and only if $\alpha>1$ or $\alpha=1$ and $0<\varrho\le1$. In the following three cases are distinguished.

   Case 1: For $1<\alpha<\infty$ the coalescent has dust. Hence, $\kappa=0$, $v(x,t)\sim e^{-\mu t}x$ and $w(x,t)\sim e^{\mu t}x$ as $x\to\infty$ with $\mu:=\int u^{-1}\Lambda({\rm d}u)=\lim_{x\to\infty}\gamma(x)/x=(\alpha/(\alpha-1))^\varrho$.
   Theorems \ref{thm_main} and \ref{thm_fix_line} are applicable and in essence logarithmic versions of \cite[Theorem 2.1]{gaisermoehle16}.

   Case 2: For $\alpha=\varrho=1$ we obtain the Bolthausen--Sznitman coalescent already studied in Example \ref{ex_beta}.

   Case 3: Assume that $\alpha=1$ and $0<\varrho<1$. Then $\kappa=0$ but nevertheless the coalescent is dust-free. Let $x>1$. By the definition (\ref{eq_gamma}) of the rate function $\gamma$,
\begin{eqnarray*}
   \gamma(x)
   & = &\frac{1}{\Gamma(\varrho)}\int_0^1 \frac{(1-u)^x-1+xu}{u}\frac{(-\log u)^{\varrho-1}}{u}\,{\rm d}u\\
& = &\frac{1}{\Gamma(\varrho+1)}\int_0^1 \frac{1-(1-u)^x-xu(1-u)^{x-1}}{u^2}(-\log u)^\varrho\,{\rm d}u,
\end{eqnarray*}
where the last equality holds by partial integration. The substitution $t=xu$ yields
\[
\frac{\gamma(x)}{x}\ =\ \frac{(\log x)^\varrho}{\Gamma(\varrho+1)}\int_0^x
\frac{1-(1-\frac{t}{x})^x-t(1-\frac{t}{x})^{x-1}}{t^2}
\bigg(1-\frac{\log t}{\log x}\bigg)^\varrho{\rm d}t.
\]
A careful analysis shows that, as $x\to\infty$, the latter integral is
asymptotically equal to $\int_0^\infty (1-e^{-t}-te^{-t})/t^2{\rm d}t+O(1/\log x)=[(e^{-t}-1)/t]_0^\infty+O(1/\log x)=1+O(1/\log x)$, which implies that
   \[
   \frac{\gamma(x)}{x}\ -\ \frac{(\log x)^\varrho}{\Gamma(\varrho+1)}
   \ =\ O((\log x)^{\rho-1})\ \to\ 0,\qquad x\to\infty.
   \]
  Proposition \ref{prop_normalizing_constant_L} (i), applied with  $\gamma_1(x):=x(\log x)^\varrho/\Gamma(\varrho+1)$, shows that the scaling $v(x,t)$ in Theorem \ref{thm_main} satisfies $v(x,t)\sim v_1(x,t)$ as $x\to\infty$, where
   $v_1(x,t)$ is the solution to the equation
   \[
   t\ =\ \int_{v_1(x,t)}^x\frac{{\rm d}u}{\gamma_1(u)}
   \ =\ \Gamma(\varrho+1)\int_{v_1(x,t)}^x \frac{{\rm d}u}{u(\log u)^\varrho}
   \ =\ \frac{\Gamma(\varrho+1)}{1-\varrho}\big((\log x)^{1-\varrho} - (\log v_1(x,t))^{1-\varrho}\big),
   \]
   whenever it exists and $v_1(x,t):=1$ otherwise. Define $C_{\varrho}:=(1-\varrho)/\Gamma(1+\varrho)$. Solving for $v_1(x,t)$ yields
   \[
   v_1(x,t)\ =\ \exp\bigg(\Big((\log x)^{1-\varrho}-C_{\varrho}t\Big)^{\frac{1}{1-\varrho}}\bigg),\qquad x>\exp\Big((C_{\varrho}t)^{(1-\varrho)^{-1}}\Big).
   \]
   By Theorem \ref{thm_main}, the process $(\log N_t^{(n)}-\log v_1(n,t))_{t\ge 0}$ converges in $D_\rz[0,\infty)$ as $n\to\infty$ to an Ornstein--Uhlenbeck type process $X$, whose generator $A^X$ satisfies (see (\ref{eq_lim_process_generator_0}))
   \[
   A^Xf(x)\ =\ \frac{1}{\Gamma(\varrho)}\int_0^1\big(f(x+\log(1-u))-f(x)+uf'(x)\big)\frac{(-\log u)^{\varrho-1}}{u^2}\,{\rm d}u,\qquad f\in D,x\in\rz,
   \]
   where $D$, the space of all twice differentiable functions $f:\rz\to\rz$ such that $f$, $f'$, $f''$ and the map $x\mapsto xf'(x)$, $x\in\rz$, belong to $\widehat{C}(\rz)$, is a core for $A^X$. Clearly, Theorem \ref{thm_fix_line} is applicable as well. Similar arguments as for the block counting process show that $w(x,t)\sim w_1(x,t)$ as $x\to\infty$, where
   \[
   w_1(x,t)\ :=\ \exp\bigg(\bigg(\frac{(1-\varrho)t}{\Gamma(\varrho+1)}+(\log x)^{1-\varrho}\bigg)^{\frac{1}{1-\varrho}}\bigg).
   \]
   Thus, the logarithmically scaled fixation line $(\log L_t^{(n)}-\log w_1(n,t))_{t\ge 0}$ converges in $D_\rz[0,\infty)$ as $n\to\infty$ to an Ornstein--Uhlenbeck type process $Y$. The generator $A^Y$ of the limiting process $Y$ satisfies
   \[
   A^Yg(y)\ =\ \frac{1}{\Gamma(\varrho)}\int_0^1\big(g(y-\log(1-u))-g(y)-ug'(y)\big)\frac{(-\log u)^{\varrho-1}}{u^2}\,{\rm d}u,\qquad g\in D,y\in\rz.
   \]

   Note that Assumption A of \cite{moehlevetter21} is not satisfied in the situation of Case 3, so both
   convergence results, for the block counting process and the fixation line, cannot be derived from the results provided in \cite{moehlevetter21}.
\end{example}

We provide another example of a dust-free $\Lambda$-coalescent which nevertheless satisfies $\kappa=0$.

\begin{example} \label{ex_particular} \rm
   Assume that the measure $\Lambda$ has density $u\mapsto 1/(1-\log u)$, $u\in (0,1)$, with respect to Lebesgue measure on $(0,1)$.
   Then, $\Lambda([0,\varepsilon])=\int_0^\varepsilon 1/(1-\log u){\rm d}u\sim e\varepsilon/(-\log\varepsilon)$ as $\varepsilon\to 0+$, and, hence,
   $\kappa=\lim_{\varepsilon\to 0+}\varepsilon^{-1}\Lambda([0,\varepsilon])=0$.
   Nevertheless, the corresponding $\Lambda$-coalescent is dust-free, since $\int u^{-1}\Lambda({\rm d}u)=[-\log(1-\log u)]_0^1=\infty$. The function $L$ in Proposition
   \ref{prop_gamma_equivalences} (iii) satisfies $L(x)=e^{\gamma(x)/x}$ and, hence, $L(x)\sim\log x$ as $x\to\infty$. Theorem \ref{thm_main} is
   applicable. By Proposition
   \ref{prop_normalizing_constant_L} (i), applied with $\gamma_1(x):=x\log\log x$, the scaling $v(x,t)$ can be chosen as the solution to the integral equation
   \[
   t=\int_{v(x,t)}^x \frac{1}{u\log\log u}{\rm d}u=
   [-{\rm Ei}(1,-\log\log u)]_{v(x,t)}^x
   ={\rm Ei}(1,-\log\log v(x,t))-{\rm Ei}(1,-\log\log x),
   \]
   where ${\rm Ei}(x):=\int_1^\infty t^{-1}e^{-xt}{\rm d}t$ denotes the
   exponential integral. By Lemma \ref{lem_generator_core_0}, the generator $A^X$ of the limiting process $X$ in Theorem \ref{thm_main} satisfies
   \[
   A^Xf(x)\ =\ \int_0^1 \frac{f(x+\log(1-u))-f(x)+uf'(x)}{u^2(1-\log u)}\,{\rm d}u,\qquad x\in\rz,f\in D,
   \]
   where $D$ denotes the space of twice differentiable functions $f:\rz\to\rz$ such that $f$, $f'$, $f''$ and the map $x\mapsto xf'(x)$, $x\in\rz$, belong to $\widehat{C}(\rz)$. We leave the formulation of the analogous results for the fixation line to the interested reader.
\end{example}

In the following an example with simultaneous multiple collisions is provided. The basic idea is to choose the measure $\Xi$ such that the corresponding $\Xi$-coalescent is dust-free, regular and stays infinite. We slightly modify the example studied in \cite{herrigermoehle} as follows.

\begin{example} \rm
Let $p_1,p_2,\ldots\in (0,1)$ with $\sum_{m=1}^\infty p_m<\infty$ and let $k_1,k_2,\ldots\in\nz$ such that $k_mp_m<1$ for all $m\in\nz$ and $\sum_{m=1}^\infty k_mp_m<\infty$. Suppose that $\Xi$ assigns for each $m\in\nz$ mass $p_m$ to the point $x^{(m)}\in\Delta$ whose first $k_m$ coordinates are equal to $p_m$ and all other coordinates are equal to $0$. Note that $\Xi(\Delta)=\sum_{m=1}^\infty p_m<\infty$. Moreover, $|x^{(m)}|=k_mp_m<1$ for all $m\in\nz$ and, hence,
$\Xi(\Delta_f)=0$ and $\Xi(\Delta^*)=0$. The corresponding $\Xi$-coalescent is dust-free, since
\[
\int_\Delta |u|\,\nu({\rm d}u)
\ =\ \sum_{m=1}^\infty |x^{(m)}| \frac{p_m}{(x^{(m)},x^{(m)})}
\ =\ \sum_{m=1}^\infty k_mp_m \frac{p_m}{k_mp_m^2}
\ =\ \sum_{m=1}^\infty 1\ =\ \infty,
\]
and regular, since $\int_\Delta |u|^2\,\nu({\rm d}u)=\sum_{m=1}^\infty
(k_mp_m)^2 p_m/(k_mp_m^2)=\sum_{m=1}^\infty k_mp_m<\infty$. Note that (see \cite[Proposition 1]{herrigermoehle}) all regular $\Xi$-coalescents are non-critical, i.e., $\nu(\Delta\setminus\Delta^\varepsilon)<\infty$ for some $\varepsilon\in (0,1)$. For all $x\ge 0$,
\begin{eqnarray*}
	\gamma(x)
	& = & \int_\Delta \sum_{i\ge 1} \big((1-u_i)^x-1+xu_i\big)
	\,\nu({\rm d}u)\\
	& = & \sum_{m=1}^\infty
	k_m \big((1-p_m)^x-1+xp_m\big)\frac{p_m}{k_mp_m^2}\\
	& = & \sum_{m=1}^\infty
	\frac{(1-p_m)^x-1+xp_m}{p_m}.
\end{eqnarray*}
Note that $\gamma(x)$ does not depend on the sequence $(k_m)_{m\in\nz}$ and is hence solely determined by the sequence $(p_m)_{m\in\nz}$.

For example, if $p_m=p^m$, $m\in\nz$, for some $p\in (0,1/2]$, then \cite[Example 6.1 b)]{herrigermoehle},
\[
\gamma(x)\ \sim\ \kappa_p x\log x,\qquad x\to\infty,
\]
with constant $\kappa_p:=-1/\log p$. Thus, $\sum_{n=2}^\infty 1/\gamma(n)=\infty$.
By Schweinsberg's criterion \cite[Proposition 33]{schweinsberg00} for non-critical coalescents, the $\Xi$-coalescent stays infinite. 

We have hence constructed a class of dust-free and regular $\Xi$-coalescents that stay infinite. For all $x>0$,
\begin{eqnarray*}
	\gamma''(x)
	& = & \sum_{m=1}^\infty \frac{(1-p^m)^x(\log(1-p^m))^2}{p^m}
	\ \sim \ \int_0^\infty \frac{(1-p^t)^x(\log(1-p^t))^2}{p^t}\,{\rm d}t\\
	& = & \kappa_p\int_0^1 \frac{(1-u)^x(\log(1-u))^2}{u^2}\,{\rm d}u
	\ \sim\ \kappa_p\int_0^1 (1-u)^x\,{\rm d}u
	\ = \ \frac{\kappa_p}{x+1},
\end{eqnarray*}
which shows that $x\gamma''(x)\to\kappa_p$ as $x\to\infty$. Thus,
Theorems \ref{thm_main} and \ref{thm_fix_line} are applicable.
\label{ex_xi}

For other choices of the sequence $(p_m)_{m\in\nz}$ one obtains further examples with different behavior. Intuitively, $\gamma(x)/x$ grows very slowly if $p_m$ tends to $0$ extremely fast. One such choice is $p_m:=p^{e^m}$ in which case we have $\gamma(x)\sim x\log\log x$ as $x\to\infty$, see also \cite[Example 6.1 c)]{herrigermoehle}. In this case $\kappa:=\lim_{x\to\infty}x\gamma''(x)=0$. Nevertheless the coalescent is dust-free. Theorems \ref{thm_main} and \ref{thm_fix_line} are applicable with scalings $v(x,t)=xL_t(x)$ and $w(x,t)=xL_t^{\#}(x)$, where $L_t$ and $L_t^{\#}$ are the slowly varying functions from Propositions \ref{prop_normalizing} and \ref{prop_normalizing_fix_line_kappa}, respectively.
\end{example}
We end this section by providing a simple example of a $\Lambda$-coalescent that does not come down from infinity but nevertheless has curvature parameter $\kappa=\infty$.
\begin{example} \rm
	Let $\alpha>0$. Consider a $\Lambda$-coalescent such that
	$\gamma(x)\sim x(\log x)(\log\log x)^\alpha$. Such a
	$\Lambda$-coalescent can be easily constructed. By Cauchy's
	condensation test, the series $\sum_{n=2}^\infty 1/\gamma(n)$
	converges if and only if $\alpha>1$. By Schweinsberg's criterion
	this coalescent therefore comes down from infinity if and only if
	$\alpha>1$. However, $\kappa=\infty$, no matter how $\alpha>0$ is
	chosen. For $\alpha\le 1$, this coalescent does not come down from
	infinity but nevertheless satisfies $\kappa=\infty$.
\label{ex_five}
\end{example}

\subsection{Proofs}\label{sec_proofs}

\subsubsection{The function \texorpdfstring{$\gamma$}{gamma}}

\begin{proof} (of Lemma \ref{lem_gamma}) Assume first that $a=0$.
Clearly, $\gamma(0)=\gamma(1)=0$. From Bernoulli's inequality (and $\nu(\{(1,0,\ldots)\})=0$) it follows that $\gamma(x)>0$ for $x>1$. By the mean value theorem, 
	there exist $\xi_i^{(1)}\in(0,u_i)$ and $\xi_i^{(2)}\in(0,\xi_i^{(1)})$ for $x\ge 2,i\in\nz$ and $u\in\Delta$ such that $\sum_{i\ge 1}(x^{-1}((1-u_i)^x-1)+u_i)=\sum_{i\ge 1}u_i(1-(1-\xi_i^{(1)})^{x-1})=(x-1)\sum_{i\ge 1}u_i\xi_i^{(1)}(1-\xi_i^{(2)})^{x-2}\le(x-1)(u,u)$. Hence, $\gamma(x)/x=\int_{\Delta}\sum_{i\ge 1}(x^{-1}((1-u_i)^x-1)+u_i)\nu({\rm d}u)\le (x-1)\Xi_{0}(\Delta)$.

	Let $u\in\Delta$. By \cite[Lemma 4.1]{moehle21}, the map $\Phi(x):=\sum_{i\ge 1}(1-(1-u_i)^x)$, $x\ge 0$, is infinitely often differentiable on $(0,\infty)$ with derivatives
	\[
		\Phi^{(k)}(x)\ =\ -\sum_{i\ge 1}(1-u_i)^x(\log(1-u_i))^k,\qquad x>0,k\in\nz,u\in\Delta.
	\]
	Thus, 
	\begin{eqnarray*}
		\frac{{\rm d}}{{\rm d}x}\sum_{i\ge 1}\big((1-u_i)^x-1+xu_i\big)& = &\frac{{\rm d}}{{\rm d}x}\big(x|u|-\Phi(x)\big)=|u|\ +\ \sum_{i\ge 1}(1-u_i)^x\log(1-u_i)\\& = &\sum_{i\ge 1}\big((1-u_i)^x\log(1-u_i)+u_i\big)
	\end{eqnarray*}
	and
	\[
		\frac{{\rm d}^k}{{\rm d}x^k}\sum_{i\ge 1}\big((1-u_i)^x-1+xu_i\big)\ =\ -\Phi^{(k)}(x)\ =\ \sum_{i\ge 1}(1-u_i)^x(\log(1-u_i))^k,\qquad k\in\nz\setminus\{1\}.
	\]
	Note that, for every $k\in\nz$, the $k$-th derivative is bounded by $C_k(u,u)$ for some $C_k>0$.  
	Hence, it is allowed to differentiate (with respect to $x$) below the integral such that $\gamma\in C_{\infty}((0,\infty))$ with derivatives as stated in the lemma.
	
	Since $\gamma'(x)>0$ for $x>1$, the function $\gamma$ is strictly increasing on $[1,\infty)$. Since $x\mapsto x^{-1}((1-u_i)^x-1)$, $x\ge 1$, is strictly increasing for every $u\in\Delta$ and $i\in\nz$, the map $x\mapsto \gamma(x)/x=\int_{\Delta}\sum_{i\ge 1}(x^{-1}((1-u_i)^x-1)+u_i)\nu({\rm d}u)$, $x\ge 1$, is strictly increasing as well.
	
	For $a>0$ the value $a\binom{x}{2}$ is added, which shows that
the results remain valid for $a>0$ as stated in the lemma.
\hfill$\Box$\end{proof}

\begin{proof} (of Proposition \ref{prop_gamma_equivalences})
    We prove this proposition by verifying the implications
    `(i) $\Rightarrow$ (ii) $\Leftrightarrow$ (iii) $\Leftrightarrow$ (iv) $\Rightarrow$ (v)' and `(v) $\Rightarrow$ (i)'.
	Define the functions $g,L:(0,\infty)\to(0,\infty)$ via $g(x):=\exp(\gamma(x)/x)=x^{\kappa}L(x)$, $x>0$. From $g'(x)=g(x)(\gamma'(x)/x-\gamma(x)/x^2)$ and $\frac{{\rm d}}{{\rm d}x}(x\gamma'(x)-\gamma(x))=x\gamma''(x)$, $x>0$, it follows that
	\begin{equation}
		\frac{xg'(x)}{g(x)}\ =\ \gamma'(x)-\frac{\gamma(x)}{x}\ =\ \frac{1}{x}\bigg(\int_{1}^{x}u\gamma''(u)\,{\rm d}u\ +\ \gamma'(1)\bigg),\qquad x>0.
	\label{eq_local_01}
	\end{equation}
	Due to the second equality of (\ref{eq_local_01}), (i) implies (ii). The map $x\mapsto x^2g'(x)=g(x)(x\gamma'(x)-\gamma(x))$, $x>0$, 
	is nondecreasing, since
	\[
		\frac{\rm{d}}{{\rm d}x}g(x)(x\gamma'(x)-\gamma(x))\ =\ g(x)\bigg(\bigg(\gamma'(x)-\frac{\gamma(x)}{x}\bigg)^2\, +\ x\gamma''(x)\bigg)\ \ge\ 0,\qquad x>0.
	\]
	Applying \cite[Theorem 2]{lamperti58} to the function $x\mapsto g(x^{-1})$, $x>0$, hence shows that (ii) is equivalent to the regular variation of $g$ with index $\kappa$, i.e., equivalent to the slow variation of $L$. Thus, (ii) and (iii) are equivalent. Conditions (iii) and (iv) are equivalent by the definition of slow variation. Suppose that (iv) holds. It is already proven that (iv) implies (ii) such that
	\begin{eqnarray*}
		\lim_{x\to\infty}\big(\gamma'(yx)-\gamma'(x)\big)\ \overset{\rm (ii)}{=}\ \lim_{x\to\infty}\bigg(\frac{\gamma(yx)}{yx}-\frac{\gamma(x)}{x}\bigg)\ \overset{\rm (iv)}{=}\ \kappa \log y,\qquad y>0,
	\end{eqnarray*}
	and (v) holds. Finally, it is shown that (v) implies (i). By the mean value theorem, there exists $\xi=\xi(x,y)$ between $y$ and $1$ such that $\gamma'(yx)-\gamma'(x)=\gamma''(x\xi)x(y-1)$ for all $x,y>0$. Given (v), $\lim_{x\to\infty}x\gamma''(x\xi)=(\kappa\log y)/(y-1)$ for $y>0,y\ne 1$. Since $\gamma''$ is nonnegative and ultimately nonincreasing, $\limsup_{x\to\infty}x\gamma''(x)\le \lim_{x\to\infty}x\gamma''(x\xi(x,y))=(\kappa\log y)/(y-1)$ for all $y\in(0,1)$ and $\liminf_{x\to\infty}x\gamma''(x)\ge \lim_{x\to\infty}x\gamma''(x\xi(x,y))=(\kappa\log y)/(y-1)$ for all $y>1$. Letting $y\to1$ establishes (i), since $\lim_{y\to1}(\log y)/(y-1)=1$.
\hfill$\Box$
\end{proof}
\begin{proof} (of Proposition \ref{prop_gamma_equi})
   We prove the equivalence of (i) and (ii) and afterwards the equivalence of (ii) and (iii).
    
   \noindent (i) $\Leftrightarrow$ (ii):
   Let $x>0$. By monotone convergence and Fubini's theorem,
   \begin{eqnarray*}
      \gamma''(x)
      & = & \sum_{i\ge 1} \int_\Delta (1-u_i)^x(\log(1-u_i))^2\nu({\rm d}u)\\
      & = & \sum_{i\ge 1} \int_\Delta \int_{[u_i,1)} x(1-y)^{x-1}\lambda({\rm d}y)
      (\log(1-u_i))^2\nu({\rm d}u)\\
      & = & \sum_{i\ge 1} x\int_{[0,1)} (1-y)^{x-1}
      \int_\Delta 1_{[0,y]}(u_i)(\log(1-u_i))^2\nu({\rm d}u)\lambda({\rm d}y)\\
      & = & \sum_{i\ge 1} x\int_0^1 (1-y)^{x-1} F_i(y){\rm d}y
      \ = \ x\int_0^1 (1-y)^{x-1}F(y){\rm d}y\\
      & = & x\int_0^\infty e^{-xt}F(1-e^{-t})\,{\rm d}t.
   \end{eqnarray*}
   Thus, the map $x\mapsto\omega(x):=\gamma''(x)/x$, $x\ge0$,
   is the Laplace transform of the measure on $(0,\infty)$ with density $u(t):=F(1-e^{-t})$, $t>0$. Assume now first that
   $0<\kappa<\infty$. By \cite[p.~445, Theorem 3]{feller}, applied with $\rho:=2$ and $L\equiv\kappa$, the asymptotics $\gamma''(x)\sim\kappa/x$, i.e., $\omega(x)\sim\kappa/x^2$
   as $x\to\infty$ is equivalent to $U(t):=\int_0^t u(s){\rm d}s\sim \kappa t^2/2$ as $t\to 0+$, which in turn is equivalent to $u(t)\sim\kappa t$ as $t\to 0+$, since the density $u$ is monotone. The proof for $\kappa=0$ 
   works with similar methods.

   \noindent (ii) $\Leftrightarrow$ (iii):
   Let $\varepsilon>0$. Choose
   $\delta=\delta(\varepsilon)\in(0,1)$ sufficiently small such that
   $(\log(1-x))^2\le(1+\varepsilon)x^2$ for all $x\in[0,\delta]$. For all
   $i\in\nz$ and $t\in[0,\delta]$ it follows that
   \[
   F_i(t)
   \ =\ \int_\Delta 1_{[0,t]}(u_i)(\log(1-u_i))^2\nu({\rm d}u)
   \ \le\ (1+\varepsilon)\int_\Delta 1_{[0,t]}(u_i)u_i^2\nu({\rm d}u)
   \ =\ (1+\varepsilon)G_i(t).
   \]
   Summation over all $i\in\nz$ yields $F(t)\le(1+\varepsilon)G(t)$ for all
   $t\in [0,\delta]$. Thus, $\liminf_{t\to 0+}t^{-1}F(t)\le(1+\varepsilon)
   \liminf_{t\to 0+}t^{-1}G(t)$ and $\limsup_{t\to 0+}t^{-1}F(t)\le(1+\varepsilon)\limsup_{t\to 0+}t^{-1}G(t)$. Since $\varepsilon>0$ can be chosen arbitrarily small it follows that
   $\liminf_{t\to 0+}t^{-1}F(t)\le\liminf_{t\to 0+} t^{-1}G(t)$ and
   $\limsup_{t\to 0+}t^{-1}F(t)\le\limsup_{t\to 0+} t^{-1}G(t)$. The converse
   two inequalities are obviously satisfied, since $G(t)\le F(t)$ for all
   $t\in[0,1)$. The equivalence of (ii) and (iii) now follows immediately.

   For $\Lambda$-coalescents it is easily seen that $G(t)=\Lambda([0,t])$, $t\in[0,1]$, showing that (iii) reduces to (\ref{lambda}).
   \hfill$\Box$
\end{proof}

\subsubsection{The normalizing function \texorpdfstring{$v$}{v}}\label{section_v}

Recall that $v(x,t)$ is the solution to $\int_{v(x,t)}^{x}(\gamma(u))^{-1}{\rm d}u=t$ for $x>1$ and $t\ge0$.

\begin{proof} (of Proposition \ref{prop_res_normalizing})
	In order to see that $v$ is well defined fix $x>1$ and define $F_x:(1,x]\to[0,\infty)$ via $F_x(y):=\int_y^x(\gamma(u))^{-1}{\rm d}u$, $y\in(1,x]$. Then $F_x(y)>0$ for $y\in(1,x]$, since $\gamma(u)>0$ for $u>1$, and $F_x\in C_1((1,x])$ with $F_x'(y)=-(\gamma(y))^{-1}$, $y\in(1,x]$, since $\gamma$ is continuous. In particular, $F_x$ is strictly decreasing. Clearly, $F_x(x)=0$. There exists $C\in\rz$ such that $\gamma(u)\le u(u-1)C$ for $u>1$. Then,
	\[
		F_x(y)\ \ge\ \frac{1}{C}\int_{y}^{x}\frac{{\rm d}u}{u(u-1)}\ =\ \frac{1}{C}\log\frac{1-x^{-1}}{1-y^{-1}},\qquad y\in(1,x],
	\]
	such that $\lim_{y\to1+}F_x(y)=\infty$. By the intermediate value theorem, the solution $v(x,t)\in(1,x]$ to the equation $F_x(v(x,t))=t$ exists and is unique for every $t\ge0$. Since $F_x'(y)<0$ for $y\in(1,x]$, the function $F_x$ is injective and the inverse function $F_x^{-1}:[0,\infty)\to(1,x]$ exists and is differentiable with $(F_x^{-1})'(t)=-\gamma(F_x^{-1}(t))$, $t\ge0$. Hence, $t\mapsto v(x,t)=F_x^{-1}(t)$, $t\ge0$, is differentiable and
	\[
		\frac{{\rm d}}{{\rm d}t}v(x,t)\ =\ -\gamma(v(x,t)),\qquad t\ge0.
	\]
	Differentiating both sides of the integral equation in (\ref{eq_integral_normalizing}) with respect to $x$ leads to $(\gamma(x))^{-1}-\frac{{\rm d}}{{\rm d}x}v(x,t)(\gamma(v(x,t)))^{-1}=0$. Equivalently,
	\[
		\frac{{\rm d}}{{\rm d}x}v(x,t)\ =\ \frac{\gamma(v(x,t))}{\gamma(x)},\qquad x>1,t\ge0.
	\]
	The two monotonicity statements follow from the formulas for the derivatives (and can also be deduced directly from Eq. (\ref{eq_integral_normalizing})). 
\hfill$\Box$\end{proof}

\begin{proof} (of Proposition \ref{prop_normalizing})
	(i) Suppose that $\Xi(\{0\})=0$ and $\mu:=\int_{\Delta}|u|\nu({\rm d}u)<\infty$. Fix $t>0$ and let $\varepsilon>0$ be arbitrary. Due to $\lim_{x\to\infty}\gamma(x)/x=\mu$, there exists $x_0>1$ such that $(\mu u)/\gamma(u)\in(1-\varepsilon,1+\varepsilon)$ for all $u\in(v(x,t),x)$ as long as $x\ge x_0$. Thus,
	\[
		\frac{1-\varepsilon}{\mu}\int_{v(x,t)}^{x}\frac{{\rm d}u}{u}\ \le\ \int_{v(x,t)}^{x}\frac{\rm{d}u}{\gamma(u)}\ \le\ \frac{1+\varepsilon}{\mu}\int_{v(x,t)}^{x}\frac{{\rm d}u}{u},
	\]
	such that, by Eq. (\ref{eq_integral_normalizing}), $\exp(-\mu t/(1-\varepsilon))\le v(x,t)/x\le \exp(-\mu t/(1+\varepsilon))$ for $x\ge x_0$. It follows that $v(x,t)\sim xe^{-\mu t}$ as $x\to\infty$, since $\varepsilon>0$ can be chosen arbitrarily small. Clearly, $v(x,0)=x$, so the statement also holds for $t=0$.
	
	(ii) Define $F:(1,\infty)\to(0,\infty)$ via $F(y)=\int_{y}^{\infty}(\gamma(u))^{-1}{\rm d}u$, $y>1$, and suppose that $F(y)<\infty$ for some (and hence all) $y>1$. Similarly to the proof of (i), it follows that $\lim_{y\to1+}F(y)=\infty$, $\lim_{y\to\infty}F(y)=0$, $F\in C_1((1,\infty))$ and $F'(y)=-(\gamma(y))^{-1}$, $y>1$. Thus, the solution $v(t)$ to the equation $F(v(t))=t$ exists and is unique for every $t\ge0$. The limit $c(t):=\lim_{x\to\infty}v(x,t)$ exists for every $t>0$, since $x\mapsto v(x,t)$, $x\ge1$, is nondecreasing, and $c(t)<\infty$ due to $\lim_{y\to\infty}F(y)=0$. From
	\[
		\int_{v(x,t)}^{v(t)}\frac{{\rm d}u}{\gamma(u)}\ =\ \int_{v(x,t)}^{x}\frac{{\rm d}u}{\gamma(u)}\ -\ \int_{v(t)}^{\infty}\frac{{\rm d}u}{\gamma(u)}\ +\ \int_{x}^{\infty}\frac{\rm{d}u}{\gamma(u)}\ =\ \int_{x}^{\infty}\frac{{\rm d}u}{\gamma(u)},\quad x\ge1,t\ge0,
	\]
	we obtain that $F(c(t))-F(v(t))=\lim_{x\to\infty}(F(v(x,t))-F(v(t)))=\lim_{x\to\infty}F(x)=0$. Since $F$ is injective, $c(t)=v(t)$ for each $t>0$. The proof of (ii) is complete. 
	
	(iii) Due to $v(x,0)=x$, the claim is true for $t=0$ with $L_0(x)=1$ for $x\ge1$. Fix $t>0$. By assumption and Proposition \ref{prop_gamma_equivalences}, there exists a slowly varying function $L:(0,\infty)\to(0,\infty)$ such that $\gamma(x)=\kappa x\log x+x\log L(x)$ for $x>0$. The fact that $L(x)=o(x^{\varepsilon})$ and $L(x)=\omega(x^{-\varepsilon})$ for every $\varepsilon>0$ is repeatedly used in this proof.
	
	First suppose that $\kappa>0$ and let $0<\varepsilon<\kappa$ be arbitrary. Recall that $\lim_{x\to\infty}v(x,t)=\infty$. There exists $x_0>1$ such that $(\kappa-\varepsilon)u\log u \le\gamma(u)\le (\kappa+\varepsilon)u\log u$ for every $u\in(v(x,t),x)$ and $x\ge x_0$. Thus,
	\[
		\frac{1}{\kappa+\varepsilon}\int_{v(x,t)}^{x}\frac{{\rm d}u}{u\log u}\ \le\ \int_{v(x,t)}^{x}\frac{{\rm d}u}{\gamma(u)}\ \le\ \frac{1}{\kappa-\varepsilon}\int_{v(x,t)}^{x}\frac{{\rm d}u}{u\log u},\qquad x\ge x_0.	
	\]
	Computing the integrals on both sides and using Eq. (\ref{eq_integral_normalizing}) yields
	\[
		\frac{1}{\kappa+\varepsilon}\log\bigg(\frac{\log x}{\log v(x,t)}\bigg)\ \le\ t\ \le\ \frac{1}{\kappa-\varepsilon}\log\bigg(\frac{\log x}{\log v(x,t)}\bigg),
	\]
	or equivalently, $x^{e^{-(\kappa+\varepsilon) t}}\le v(x,t)\le x^{e^{-(\kappa-\varepsilon)t}}$ for all $x\ge x_0$. From (\ref{eq_res_normalizing_derivative}) it follows that
	\[
		\frac{\frac{{\rm d}}{{\rm d}x}v(x,t)x}{v(x,t)}\ =\ \frac{\gamma(v(x,t))}{v(x,t)}\frac{x}{\gamma(x)}\ =\ \frac{\kappa \log v(x,t)+\log L(v(x,t))}{\kappa \log x+\log L(x)},\qquad x>1,
	\]
	such that
	\begin{equation}
		e^{-(\kappa+\varepsilon) t}\ \le\ \liminf_{x\to\infty}\frac{\frac{{\rm d}}{{\rm d}x}v(x,t)x}{v(x,t)}\ \le\ \limsup_{x\to\infty}\frac{\frac{{\rm d}}{{\rm d}x}v(x,t)x}{v(x,t)}\ \le\ e^{-(\kappa-\varepsilon)t}.
	\label{eq_proof_norm_local_01}
	\end{equation}
	
	We are going to show similar inequalities for $\kappa=0$. Suppose that $\gamma(x)=x\log L(x)$, $x>0$, for some slowly varying function $L$. By Proposition \ref{prop_gamma_equivalences}, $(xL'(x))/L(x)=\gamma'(x)-\gamma(x)/x\to\kappa=0$ as $x\to\infty$. From (\ref{eq_res_normalizing_derivative}) it follows that
	\begin{eqnarray*}
		\frac{{\rm d}}{{\rm d}t}\log L(v(x,t))\ =\ \frac{-L'(v(x,t))v(x,t)\log L(v(x,t))}{L(v(x,t))},\qquad x>1.
	\end{eqnarray*}
	Note that $L(x)=\exp(\gamma(x)/x)$ is nondecreasing on $[1,\infty)$ and, by definition, $v(x,t)\le x$ for all $x\ge1$. For $\varepsilon>0$ there exists $x_1>1$ such that $|\frac{{\rm d}}{{\rm d}s}\log L(v(x,s))|\le \varepsilon\log L(x)$ for all $s\in[0,t]$ and $x\ge x_1$. We hence obtain
	\[
		\big\vert\log L(v(x,t))-\log L(x)\big\vert\ \le\ \int_{0}^{t}\bigg|\frac{{\rm d}}{{\rm d}s}\log L(v(x,s))\bigg|\,{\rm d}s\ \le\ \varepsilon t\log L(x)
	\]
	such that
	\begin{equation}
		\frac{\frac{{\rm d}}{{\rm d}x}v(x,t)x}{v(x,t)}\ =\ \frac{\log L(v(x,t))}{\log L(x)}\ \in\ [1-\varepsilon t,1+\varepsilon]
	\label{eq_proof_norm_local_02}
	\end{equation}
	for all $x\ge x_1$. Letting $\varepsilon\to0$ in (\ref{eq_proof_norm_local_01}) and (\ref{eq_proof_norm_local_02}) yields $\lim_{x\to\infty}\frac{\frac{{\rm d}}{{\rm d}x}v(x,t)x}{v(x,t)}=e^{-\kappa t}$ for $\kappa\ge0$. A `variant at infinity' of \cite[Theorem 2]{lamperti58} completes the proof.\hfill$\Box$
\end{proof}

\begin{proof} (of Proposition \ref{prop_normalizing_constant_L})
	(i) Define the function $L_1:(1,\infty)\to(0,\infty)$ via $\gamma_1(x)=\kappa x\log x+x\log L_1(x)$, $x>1$. By assumption, 
	\[
		r(x)\ :=\ \frac{\gamma(x)-\gamma_1(x)}{x}\ =\ \log\bigg(\frac{L(x)}{L_1(x)}\bigg)\ \to\ 0,\qquad x\to\infty.
	\]
	In particular, $L(x)\sim L_1(x)$ as $x\to\infty$. Hence, $L_1$ is slowly varying and, as a consequence, $\gamma_1$ satisfies (\ref{eq_res_L}). Unfortunately, the scaling $v_1(x,t)$, defined by the integral equation in (\ref{eq_integral_normalizing}) with $\gamma_1$ in place of $\gamma$, does not exist globally. The reason is that the condition $\lim_{y\to1+}F_x(y)=\infty$ from the proof of Proposition \ref{prop_res_normalizing} cannot be guaranteed. However, $\gamma_1$ is continuous and positive on $(1,\infty)$ and $\int_{c}^{\infty}(\gamma_1(u))^{-1}{\rm d}u=\infty$ for each $c>1$. Carefully reading the proof of Proposition \ref{prop_res_normalizing} shows that the statements of Proposition \ref{prop_res_normalizing} remain true with the restriction that, for each $t\ge 0$, $x\ge x_0(t)$ for some $x_0(t)>1$. Moreover, we can choose $x_0(t)$ in such a way that $x_0(s)\le x_0(t)$ for $s\le t$. In particular, the scaling $v_1(x,t)$ exists for $x\ge x_0(t)$ and $t\ge0$. Now fix $t\ge 0$. From (\ref{eq_res_normalizing_derivative}) it follows that
	\[
		\log v(x,t)-\log x\ =\ -\int_{0}^{t}\frac{\gamma(v(x,s))}{v(x,s)}\,{\rm d}s
		\ =\ -\int_{0}^{t}\big(\kappa\log v(x,s)+\log L(v(x,s)\big)\,{\rm d}s
	\]
	for $x>1$. The same equalities hold when $x\ge x_0(t)$, and $v(x,t)$ and $L$ are replaced by $v_1(x,t)$ and $L_1$, respectively. Then, for $x>x_0(t)$,
	\[
		\bigg|\log\frac{v(x,t)}{v_1(x,t)}\bigg|
		\ \le\ \kappa\int_{0}^{t}\bigg|\log\frac{v(x,s)}{v_1(x,s)}\bigg|\,{\rm d}s\ +\ t\sup_{y\ge v(x,t)}|r(y)|\ +\ \int_{0}^{t}\bigg|\log\frac{L_1(v(x,s))}{L_1(v_1(x,s))}\bigg|\,{\rm d}s.
	\]
	Let $c_1,c_2>0$ be arbitrary. The representation theorem for slowly varying functions \cite[Theorem 1.3.1]{binghamgoldieteugels} states the existence of functions $\varepsilon,\delta:(0,\infty)\to\rz$ with $\lim_{x\to\infty}\varepsilon(x)=0$ and $\lim_{x\to\infty}\delta(x)=d \in\rz$ such that $\log L_1(x)=\delta(x)+\int_{1}^{x}(\varepsilon(u)/u){\rm d}u$, $x>0$. Hence,
	\[
		\bigg|\log\frac{L_1(v(x,s))}{L_1(v_1(x,s))}\bigg|\ \le\ \big\vert\delta(v(x,s))-\delta(v_1(x,s))\big\vert+\int_{v_1(x,s)}^{v(x,s)}\frac{|\varepsilon(u)|}{u}\,{\rm d}u\ \le\ c_2+c_1\bigg|\log\frac{v(x,s)}{v_1(x,s)}\bigg|
	\]
	for sufficiently large $x$ and $s\in[0,t]$, where in the last inequality it is used that $\inf_{s\in[0,t]}v(x,s)=v(x,t)\to\infty$ and $\inf_{s\in[0,t]}v_1(x,s)\to\infty$ as $x\to\infty$. Thus,
	\[
		\bigg|\log\frac{v(x,t)}{v_1(x,t)}\bigg|\ \le\ t\bigg(\sup_{y\ge v(x,t)}|r(y)|+c_2\bigg)\ +\ (\kappa+c_1)\int_{0}^{t}\bigg|\log\frac{v(x,s)}{ v_1(x,s)}\bigg|\,{\rm d}s.
	\]
	By Gronwall's inequality,
	\[
		\bigg|\log\frac{v(x,t)}{v_1(x,t)}\bigg|\ \le\ t\bigg(\sup_{y\ge v_1(x,t)}|r(y)|+c_2\bigg)\exp(t(\kappa +c_1)).
	\]
	We conclude that $\lim_{x\to\infty}|\log v(x,t)-\log v_1(x,t)|=0$, which completes the proof of (i), because $c_2>0$ is arbitrarily small.

	(ii) Assume first that $\gamma:(0,\infty)\to(0,\infty)$ is a function of the form (\ref{eq_res_L}) with $L\equiv C$ for some constant $C>0$. Then the integral in (\ref{eq_integral_normalizing}) can be calculated explicitly and it is easily seen that $v(x,t)$, defined via $v(x,t):=x^{e^{-\kappa t}}C^{\kappa^{-1}(e^{-\kappa t}-1)}$ if $\kappa>0$ and $v(x,t):=xC^{-t}$ if $\kappa=0$, solves (\ref{eq_integral_normalizing}) for every $x>1$ and $t\ge0$. If $L$ only satisfies $L(x)\to C$ as $x\to\infty$, then the same formulas for $v(x,t)$ hold, but with equality replaced by asymptotic equality as $x\to\infty$, as shown in (i). \hfill$\Box$
\end{proof}

\subsubsection{Proof of Theorem \ref{thm_main}}
The scaled block counting process and the scaled fixation line are in general time-inhomogeneous Markov processes. We therefore add a further `time variable' and consider the associated time-space processes, which are time-homogeneous. We want to show the uniform convergence of the generators. First a well-known result (\cite[Theorem 3.1]{satoyamazato84}) concerning generators of Ornstein--Uhlenbeck type processes on $\rz^d$ is applied. The short proof is an adaption of the proof of \cite[Lemma 6]{moehlevetter21} to the $\Xi$-coalescent setting.

\begin{lemma}
	Suppose that $\Xi$ satisfies (\ref{eq_regularity}) and $\Xi(\Delta^*)=0$. Fix $\kappa\in[0,\infty)$ and let the family of operators $(T_t^X)_{t\ge0}$ be defined by (\ref{eq_semigroup_lim}). Then $(T_t^X)_{t\ge0}$ is a Feller semigroup on $\widehat{C}(\rz)$. Let $D$ denote the space of all twice differentiable functions $f:\rz\to\rz$ such that $f$, $f'$, $f''$ and the map $x\mapsto xf'(x)$, $x\in\rz$, belong to $\widehat{C}(\rz)$. Then $D$ is a core for the generator $A^X$ corresponding to $(T_t^X)_{t\ge0}$ and
    \begin{equation}
	A^Xf(x)\ =\ -\kappa xf'(x)\ +\ \int_{\Delta}\big(f(x+\log(1-|u|))-f(x)+|u|f'(x)\big)\nu({\rm d}u),\quad x\in\rz,f\in D.
	\label{eq_lim_process_generator_0}
	\end{equation}
	\label{lem_generator_core_0}
\end{lemma}
\begin{proof} (of Lemma \ref{lem_generator_core_0})
	Substituting $g:\Delta\setminus\Delta^*\to\rz$, $g(u):=\log(1-|u|)$, $u\in\Delta\setminus\Delta^*,$ shows that (\ref{eq_lim_process_generator_0}) is an integro-differential operator of the form (1.1) of Sato and Yamazato \cite{satoyamazato84} with dimension $d=1$. In \cite{satoyamazato84}, operators of this form are initially considered as acting on the space $C_c^2$ of twice differentiable functions with compact support (see the explanations after Eq.~(1.2) in \cite{satoyamazato84}), but Step 3 of the proof of \cite[Theorem 3.1]{satoyamazato84} shows that (\ref{eq_lim_process_generator_0}) even holds for functions $f\in D$ ($\supset C_c^2$). Note that the space $D$ is denoted by $F_1$ in \cite{satoyamazato84}. The fact that $D$ is a core for $A^X$ is only a different phrasing of the claim in Step 5 of the proof of \cite[Theorem 3.1]{satoyamazato84}.
\hfill$\Box$\end{proof}

When writing semigroups or generators in the remainder of the proof section, we mostly omit the upper index that identifies the corresponding process. We only use the symbol tilde to indicate the time-space process.

The time-space process $\widetilde{X}:=(t,X_t)_{t\ge0}$ is a time-homogeneous Markov process with state space $\widetilde{E}:=[0,\infty)\times\rz$ and semigroup $\widetilde{T}:=(\widetilde{T}_t)_{t\ge0}$, given by
\[
	\widetilde{T}_tf(s,x)\ :=\ \me(f(s+t,e^{-\kappa t}x+S_t)),\qquad (s,x)\in\widetilde{E},f\in B(\widetilde{E}),t\ge0.
\]
For $f\in\widehat{C}(\widetilde{E})$ and $s\ge0$ let the map $x\mapsto f(s,x)$, $x\in\rz$, be denoted by $\pi f(s,x)$. Let $\widetilde{D}$ denote the space of functions $f\in\widehat{C}(\widetilde{E})$ of the form $f(s,x)=\sum_{i=1}^{l}g_i(s)h_i(x)$, $(s,x)\in\widetilde{E}$, with $l\in\nz, h_i\in D$ and $g_i\in C_1([0,\infty))$ such that $g_i,g_i'\in\widehat{C}([0,\infty))$ for $i\in\{1,\ldots,l\}$. By \cite[Proposition 10]{moehlevetter21}, $\widetilde{T}$ is a Feller semigroup, $\widetilde{D}$ is a core for the generator $\widetilde{A}$ corresponding to $\widetilde{T}$ and
\[
	\widetilde{A}f(s,x)\ =\ \frac{\partial}{\partial s}f(s,x)\ +\ A^X \pi f(s,x),\qquad (s,x)\in\widetilde{E},f\in\widetilde{D}.
\]
For $n\in\nz$ the logarithmically scaled block counting process $X^{(n)}:=(X_t^{(n)})_{t\ge0}:=(\log N_t^{(n)}-\log v(n,t))_{t\ge0}$ is a time-inhomogeneous Markov process. The random variable $X_s^{(n)}$ takes values in $E_{n,s}:=\{x\in\rz : e^xv(n,s)\in[n]\}$. The `generator' $(A_s^{(n)})_{s\ge0}$ of $X^{(n)}$ is given by
\begin{eqnarray*}
	A_s^{(n)}f(x)& = &f'(x)\frac{-\tfrac{{\rm d}}{{\rm d}s}v(n,s)}{v(n,s)}\ +\ \sum_{j=1}^{xv(n,s)-1}\big(f(\log j -\log v(n,s))-f(x)\big)q_{xv(n,s),j}\\
	& = &f'(x)\frac{\gamma(v(n,s))}{v(n,s)}\ +\ \sum_{j=1}^{xv(n,s)-1}\big(f(\log j -\log
	v(n,s))-f(x)\big)q_{xv(n,s),j}
\end{eqnarray*}
for $x\in E_{n,s}$ and $s\ge0$. Here $f\in C_1(\rz)$ such that $f, f'\in\widehat{C}(\rz)$. The time-space process $\widetilde{X}^{(n)}:=(t,X_t^{(n)})_{t\ge0}$ is a time-homogeneous Markov process with state space $\widetilde{E}_n:=\{(s,x)\in[0,\infty)\times\rz : e^xv(n,s)\in[n]\}$ and semigroup $\widetilde{T}^{(n)}:=(\widetilde{T}_t^{(n)})_{t\ge0}$, given by
\[
	\widetilde{T}_t^{(n)}f(s,x)\ :=\ \me\big(f(s+t,\log N_t^{(e^xv(n,s))}-\log v(n,s+t))\big),\qquad (s,x)\in\widetilde{E}_n,f\in B(\widetilde{E}),t\ge0.	
\]
For $f\in\widetilde{D}$ (restricted to $\widetilde{E}_n\subset\widetilde{E}$) the corresponding generator $\widetilde{A}^{(n)}$ is given by
\[
	\widetilde{A}^{(n)}f(s,x)\ =\ \frac{\partial}{\partial s}f(s,x)\ +\ A_s^{(n)}\pi f(s,x),\qquad (s,x)\in\widetilde{E}_n.
\]

\begin{proof} (of Theorem \ref{thm_main})
	Write $k:=k(s,x,n):=e^xv(n,s)$ for $(s,x)\in\widetilde{E}_n$ and $n\in\nz$. Let $h\in D$. Define $R(k,x):=\frac{\gamma(ke^{-x})}{ke^{-x}}-\gamma(k)/k$ and
	\[
		S(k,x)\ :=\ \sum_{j=1}^{k-1}\big(h(x+\log\tfrac{j}{k})-h(x)+(1-\tfrac{j}{k})h'(x)\big)q_{k,j},\qquad k\in\nz,x\in\rz,
	\]
	such that
	\[
		A_s^{(n)}h(x)\ =\ h'(x)R(k,x)\ +\ S(k,x),\qquad (s,x)\in\widetilde{E}_n,n\in\nz.
	\]
 	Define the continuous function $I:\rz\times[0,1]\to\rz$ via $I(x,y):=h(x+\log(1-y))-h(x)+yh'(x)$, $y\in[0,1)$, and $I(x,1):=-h(x)+h'(x)$ for $x\in\rz$. From Eq.~(\ref{eq_rates_block_counting}) and the definition of $I$ it follows that
 	\[
 		S(k,x)
 		\ =\ \int_{\Delta}\me(I(x,1-Y(k,u)/k))\,\nu({\rm d}u),\qquad k\in\nz,x\in\rz.
 	\]
 	Also,
 	\[
 		A^Xh(x)\ =\ -\kappa xh'(x)\ +\ \int_{\Delta}I(x,|u|)\,\nu({\rm d}u),\qquad x\in\rz.
 	\]
 	Part 1 of the proof treats the convergence of $R(k,x)$ and Part 2 the convergence of $S(k,x)$.
 	
	\textbf{Part 1.} By assumption and Proposition \ref{prop_gamma_equivalences}, there exist $\kappa\ge0$ and a slowly varying function $L:(0,\infty)\to(0,\infty)$ such that $\gamma(x)=\kappa x\log x+x\log L(x)$, $x>0$. Then $R(k,x)+\kappa x=\log(L(ke^{-x})/L(k))$ for $k\in\nz$ and $x\in\rz$. 
	Applying \cite[Theorem 1.5.6 (ii)]{binghamgoldieteugels}, a boundary for the growth of slowly varying functions, 
	yields the existence of $C>0$ such that $|R(k,x)+\kappa x)|\le C+|x|$ for $k\in\nz$ and $-\infty<x\le \log k$. For $c>0$ there exist $-\infty<K_1<K_2<\infty$ such that
	\begin{equation}
		|h'(x)(R(k,x)+\kappa x)|\ \le\ C|h'(x)|+|xh'(x)|\ \le\ c, \qquad x\in\rz\setminus[K_1,K_2],x\le\log k,k\in\nz, 	
	\label{eq_proof_main_local_1}	
	\end{equation}
	since $h'$ and the map $x\mapsto xh'(x)$, $x\in\rz$, vanish as $|x|\to\infty$.
	The present restriction $x\le\log k(s,x,n)$ is met for $(s,x)\in\widetilde{E}_n$ and $n\in\nz$. Let $T>0$ be arbitrary. By the uniform convergence theorem for slowly varying functions (\cite[Theorem 1.5.2]{binghamgoldieteugels}) and $\lim_{n\to\infty}\inf_{s\in[0,T]}v(n,s)=\infty$,
	\begin{eqnarray}
		&&\lim_{n\to\infty}\sup_{(s,x)\in\widetilde{E}_n,s\in[0,T],x\in[K_1,K_2]}|R(k,x)+\kappa x|\nonumber\\&&~~~\ =\ \lim_{n\to\infty}\sup_{(s,x)\in\widetilde{E}_n,s\in[0,T],x\in[K_1,K_2]}|\log(L(v(n,s))/L(e^xv(n,s)))|\ =\ 0.
	\label{eq_proof_main_local_2}
	\end{eqnarray}
	From (\ref{eq_proof_main_local_1}), (\ref{eq_proof_main_local_2}) and arbitrariness of $c$ it follows that
	\begin{equation}
		\lim_{n\to\infty}\sup_{(s,x)\in\widetilde{E}_n,s\in[0,T ]}|h'(x)(R(k,x)+\kappa x)|\ =\ 0.
	\label{eq_Part_1}
	\end{equation}
	
	\textbf{Part 2.} Note that, as $n\to\infty$, $k=e^xv(n,s)\to\infty$ or $x\to-\infty$. For example, either $k\ge\sqrt{v(n,T)}$ or $x<-\tfrac{1}{2}\log v(n,T)$ for each $(s,x)\in\widetilde{E}_n$ with $s\in[0,T]$ and $n\in\nz$. In order to prove that
	\begin{equation}
		\lim_{n\to\infty}\sup_{(s,x)\in\widetilde{E}_n,s\in[0,T]}\big|\me\big(I(x,1-Y(k,u)/k)-I(x,|u|)\big)\big|\ =\ 0,\qquad u\in\Delta\setminus(\Delta^*\cup\{0\}),
	\label{eq_local}
	\end{equation}
	it therefore suffices to show that $\lim_{x\to-\infty}I(x,|u|)=0$, $\lim_{x\to-\infty}\me(I(x,Y(k,u)/k))=0$ for any $k\in\nz$ and $\lim_{k\to\infty}\sup_{x\in\rz}|\me(I(x,1-Y(k,u)/k)-I(x,|u|))|=0$  for each $u\in\Delta\setminus(\Delta^*\cup\{0\})$.

	Clearly, $\sup_{x\in\rz,y\in[0,1]}|I(x,y)|\le 2||h||+||h'||<\infty$. In particular, the family of functions $\mathcal{I}:=\{I(x,.):x\in\rz\}$ is uniformly bounded. The family $\mathcal{I}$ is equicontinuous on any interval $[0,c]$ with $c<1$, since $h$ is uniformly continuous and $h'$ is bounded. In view of \cite[Lemma 9]{moehlevetter21}, the almost sure convergence of $1-Y(k,u)/k$ to $|u|$ as $k\to\infty$ implies that $\lim_{k\to\infty}\sup_{x\in\rz}|\me(I(x,1-Y(k,u)/k))-I(x,|u|)|=0$ for any $u\in\Delta\setminus(\Delta^*\cup\{0\})$. 
	The cited lemma does not allow the limiting `random' variable $|u|$ to obtain the values $0$ and $1$ with positive probability, hence we impose the restriction of $u$ to $\Delta\setminus(\Delta^*\cup \{0\})$. For any $y\in[0,1]$, $\lim_{x\to-\infty}I(x,y)=0$, since $\lim_{|x|\to\infty}h'(x)=\lim_{|x|\to\infty}h(x)=0$. Thus, $\lim_{x\to-\infty}I(x,|u|)=0$ and, by dominated convergence, $\lim_{x\to-\infty}\me(I(x,1-Y(k,u)/k))$ for $k\in\nz$ and $u\in\Delta$, which completes the proof of (\ref{eq_local}).
		
	 Taylor's theorem applied to $y\mapsto h(x+\log(1-y))$, $y<1$, evaluated at $y=0$ with mean value remainder states the existence of $\xi\in
	(0,y)$ such that
	\[
		I(x,y)\ =\ (1-\xi)^{-2}(h''(x+\log(1-\xi))-h'(x+\log(1-\xi)))(y-\xi)y,\qquad x\in\rz,y\in(0,1).
	\]
	In particular, $\sup_{x\in\rz}|I(x,y)|\le (1-c)^{-2}(||h''||-||h'||)y^2<\infty$ for $0\le y\le c$ and any $c<1$. Thus, there exists $C\in\rz$ such that $\sup_{x\in\rz}|I(x,y)|\le Cy^2$ for every $y\in[0,1]$. From Lemma \ref{lem_bound_Y} it follows that $\sup_{n\in\nz}{\sup_{(s,x)\in\widetilde{E}_n,s\in[0,T ]}}|\me(I(x,1-Y(k,u)/k))-I(x,|u|)|\le\sup_{k\in\nz,x\in\rz}|\me(I(x,1-Y(k,u)/k))|+\sup_{x\in\rz}|I(x,|u|)|\le (D_2+1)C|u|^2$ for any $u\in\Delta$. 
	Due to (\ref{eq_local}) and (\ref{eq_regularity}), the dominated convergence theorem is applicable such that
	\begin{eqnarray}
		&&\lim_{n\to\infty}\sup_{(s,x)\in\widetilde{E}_n,s\in[0,T]}\bigg|\int_{\Delta}\me(I(x,1-Y(k,u)/k))\,\nu({\rm d}u)-\int_{\Delta}I(x,|u|)\,\nu({\rm d}u)\bigg|
		\nonumber\\&&~~~\ \le\ \lim_{n\to\infty}\int_{\Delta}\sup_{(s,x)\in\widetilde{E}_n,s\in[0,T ]}\big|\me(I(x,1-Y(k,u)/k))-I(x,|u|)\big|\,\nu({\rm d}u)\ =\ 0.
		\label{eq_Part_2}
	\end{eqnarray}
	Here we made use of $\nu(\Delta^*\cup\{0\})=0$.
	
	Eqs. (\ref{eq_Part_1}) and (\ref{eq_Part_2}) imply
	\[
		\lim_{n\to\infty}\sup_{(s,x)\in\widetilde{E}_n,s\in[0,T]}|A_s^{(n)}h(x)-A^Xh(x)|\ =\ 0.
	\]
	Hence,
	\[
		\lim_{n\to\infty}\sup_{(s,x)\in\widetilde{E}_n,s\in[0,T ]}|\widetilde{A}^{(n)}f(s,x)-\widetilde{A}f(s,x)|\ =\ 0,\qquad f\in\widetilde{D}.
	\]
	From \cite[IV, Corollary 8.7]{ethierkurtz} it follows that $\widetilde{X}^{(n)}\to\widetilde{X}$ in $D_{\widetilde{E}}[0,\infty)$, hence $X^{(n)}\to X$ in $D_{\rz}[0,\infty)$ as $n\to\infty$.
\hfill$\Box$\end{proof}

\begin{remark}
	Assumption (\ref{eq_res_kappa}) is only used in Part 1 of the proof of Theorem \ref{thm_main}, whereas Part 2 remains correct for all measure $\Xi$ satisfying $\Xi(\Delta^*\cup\{0\})=0$ and $\int_{2}^{\infty}(\gamma(u))^{-1}{\rm d}u=\infty$.
\end{remark}

\subsubsection{Proofs concerning the fixation line}\label{sec_proof_fix_line}

Propositions \ref{prop_normalizing_fix_line} and \ref{prop_normalizing_fix_line_kappa} treat the normalizing function $w(x,t)$ for the fixation line, implicitly defined via $\int_{x}^{w(x,t)}(\gamma(u))^{-1}{\rm d}u=t$. Proposition \ref{prop_normalizing_fix_line} verifies the existence of $w$.

\begin{proof} (of Proposition \ref{prop_normalizing_fix_line})
	Suppose that $\int_2^\infty(\gamma(u))^{-1}{\rm d}u=\infty$. Fix $x>1$. The function $F_x:[x,\infty)\to\rz$, defined by $F_x(y):=\int_{x}^{y}(\gamma(u))^{-1}{\rm d}u$, $y\in[x,\infty)$, is continuous, strictly increasing and satisfies $F_x(x)=0$ and $\lim_{y\to\infty}F_x(y)=\infty$. Thus, the solution $w(x,t)$ to the equation $t=F_x(w(x,t))=\int_{x}^{w(x,t)}(\gamma(u))^{-1}{\rm d}u$, exists, lies in the interval  $[x,\infty)$ and is unique for every $t\ge0$. The function $F_x$ is differentiable and $F_x'(y)=(\gamma(y))^{-1}>0$, $y\in[x,\infty),$ and, as a consequence, the inverse $F_x^{-1}:[0,\infty)\to[x,\infty)$ exists, is differentiable and $(F_x^{-1})'(t)=\gamma(F_x^{-1}(t))$, $t\ge0$. Clearly, $w(x,t)=F_x^{-1}(t)$ such that
	\[
		\frac{{\rm d}}{{\rm d}t}w(x,t)\ =\ \gamma(w(x,t)),\qquad t\ge0, x>1.
	\]
	The formula for $\frac{{\rm d}}{{\rm d}x}w(x,t)$ follows from differentiation of both sides of the integral equation in (\ref{eq_int_normalizing_fix_line}) with respect to $x$.
\hfill$\Box$\end{proof}

The proof of Proposition \ref{prop_normalizing_fix_line_kappa} could be copied from the respective one for the block counting process, the proof given instead uses the fact that $v(.,t)$ and $w(.,t)$ are inverse.

\begin{proof} (of Proposition \ref{prop_normalizing_fix_line_kappa})
	We first prove (ii), and then (i) and (iii). Note that the situation of (i) is a special case of (ii) with $\kappa=0$. Fix $t\ge0$. According to Proposition \ref{prop_normalizing} there exists a slowly varying function $L_t:[1,\infty)\to(0,\infty)$ such that $v(x,t)=x^{e^{-\kappa t}}L_t(x)$, $x\ge1$. As the function $w(.,t)$ is the inverse of $v(.,t)$, it is regularly varying with index $e^{\kappa t}$. More precisely, it follows from \cite[Proposition 1.5.15]{binghamgoldieteugels} applied with $f(x):=v(x,t)$, $a:=e^{-\kappa t}$, $b:=1$ and $l(x):=L_t(x),$ $x\ge 1,$ that $w(x,t)\sim x^{e^{\kappa t}}L_t^{\#,0}(x^{e^{\kappa t}})$ as $x\to\infty$, where $L_t^{\#,0}$ is the de Bruijn conjugate of the slowly varying function $x\mapsto (L_t(x))^{e^{\kappa t}}$, $x\ge1$, i.e., a slowly varying function satisfying $\lim_{x\to\infty}L_t^{\#,0}(x(L_t(x))^{e^{\kappa t}})(L_t(x))^{e^{\kappa t}}=1$. See e.g. \cite[Theorem 1.5.13]{binghamgoldieteugels} for a definition of the de Bruijn conjugate of slowly varying functions. The function $L_t^{\#}$, defined via $w(x,t)=x^{e^{\kappa t}}L_t^{\#}(x),$ $x\ge 1,$ is asymptotically equal to the slowly varying function $L_t^{\#,0}(x^{e^{\kappa t}})$, thus slowly varying itself, which completes the proof of (ii).

	(i) Assume that $\Xi(\{0\})=0$ and $\mu=\int_\Delta|u|\nu({\rm d}u)<\infty$, and recall that $\kappa=0$. Proposition \ref{prop_normalizing} states that $\lim_{x\to\infty}L_t(x)=e^{-\mu t}$. We can thus choose $L_t^{\#,0}(x):=e^{\mu t}$, $x\ge1$. From $L_t^{\#}(x)\sim L_t^{\#,0}(x)=e^{\mu t}$ it follows that $w(x,t)\sim xe^{\mu t}$ as $x\to\infty$.
	
	(iii) As seen in the proof of Proposition \ref{prop_normalizing_constant_L}, there exists a slowly varying function $L_1:(1,\infty)\to(0,\infty)$ such that $\gamma_1(x)=\kappa x\log x+ x\log L_1(x),$ $x>1$. The function $\gamma_1$ is continuous and positive on $(1,\infty)$ and $\int_{2}^{\infty}(\gamma_1(u))^{-1}{\rm d}u=\infty$. The proof of Proposition \ref{prop_normalizing_fix_line} shows that the scaling $w_1(x,t)$, defined by (\ref{eq_int_normalizing_fix_line}) with $\gamma_1$ in place of $\gamma$, exists for $x\ge 1$ and $t\ge 0$. Fix $t\ge 0$. According to Proposition \ref{prop_normalizing_constant_L}, the scaling $v_1(x,t)$, defined by the integral equation in (\ref{eq_integral_normalizing}) with $\gamma_1$ in place of $\gamma$, exists for $x\ge x_0(t)$, where $x_0(t)>1$. The proof of Proposition \ref{prop_normalizing} shows the existence of a slowly varying function $L_{t,1}$ 
	such that $v_1(x,t)=x^{e^{-\kappa t}}L_{t,1}(x)$ for $x\ge x_0(t)$. Here we used that the map $x\mapsto\gamma_1(x)/x,$ $x> 1,$ and, hence, the function $L_1$ are nondecreasing if $\kappa=0$.  The scalings $v_1(.,t)$ and $w_1(.,t)$ are obviously inverse (on suitable domains). From Part (ii) it follows that there exists a slowly varying function $L_{t,1}^{\#}$ 
	such that $w_1(x,t)=x^{e^{\kappa t}}L_{t,1}^{\#}(x)$. From $v(x,t)\sim v_1(x,t)$ it follows that $L_t(x)\sim L_{t,1}(x)$ and $x=w(v(x,t))\sim w(v_1(x,t))=x(L_t(x))^{e^{\kappa t}}L_{t}^{\#}(v_1(x,t))$ as $x\to\infty$ and $x=w_1(v_1(x,t),t)=x(L_{t,1}(x))^{e^{\kappa t}}L_{t,1}^{\#}(v_1(x,t))$ for $x\ge x_0(t)$. Hence, $L_{t,1}^{\#}(v_2(x,t))\sim L_{t,2}^{\#}(v_2(x,t))$, consequently $L_{t,1}^{\#}(x)\sim L_{t,2}^{\#}(x)$ and we finally have $w_1(x,t)\sim w_2(x,t)$ as $x\to\infty$. \hfill$\Box$
\end{proof}

We proceed to prove the convergence of the scaled fixation line. The involved state spaces and semigroups are denoted by the same symbols as for the block counting process.

\begin{proof} (of Theorem \ref{thm_fix_line})
	Define $Y_t^{(n)}:=\log L_t^{(n)}-\log w(n,t)$ for $n\in\nz$ and $t\ge0$. We start by proving the convergence of the one-dimensional distributions. Fix $t\ge0$, $x\in\rz$ and write $k:=\lceil e^xw(n,t)\rceil\in\nz$. Note that $Y_t\overset{\text{d}}{=}-e^{\kappa t}S_t$. By duality (Eq. (\ref{eq_intro_duality})),
	\[
		\pr(Y_t^{(n)}\ge x)\ =\ \pr(L_t^{(n)}\ge k)\ =\ \pr(N_t^{(k)}\le n)\ =\ \pr\big(\log N_t^{(k)}-\log v(k,t)\le \log(n/v(k,t))\big).
	\]
	By Proposition \ref{prop_normalizing}, the function $v(.,t)$ varies regularly with index $e^{-\kappa t}$. From $\lim_{n\to\infty}w(n,t)=\infty$ it hence follows that $n/v(k,t)=v(w(n,t),t)/v(\lceil e^xw(n,t)\rceil,t)\to e^{-xe^{-\kappa t}}$ as $n\to\infty$. Theorem \ref{thm_main} implies that
	\begin{equation}
		\lim_{n\to\infty}\pr(Y_t^{(n)}\ge x)\ =\ \pr(S_t\le -xe^{-\kappa t})\ =\ \pr(Y_t\ge x)
	\label{eq_proof_fix_line_local_01}
	\end{equation}
	for $-xe^{-\kappa t}$ in the set $C_{{S_t}}$ of continuity points of $S_t$. From (\ref{eq_proof_fix_line_local_01}) we obtain the weak convergence of $Y_t^{(n)}$ to $Y_t$ as $n\to\infty$ for each $t\ge0$, since $-xe^{-\kappa t}\in C_{S_t}$ if and only if $x\in C_{Y_t}$.
	
	The time-space processes $\widetilde{Y}^{(n)}:=(t,Y_t^{(n)})_{t\ge0}$, $n\in\nz$, and $\widetilde{Y}:=(t,X_t)_{t\ge0}$ are time-homogeneous Markov processes with state spaces $\widetilde{E}_{n}=\{(s,x) : s\ge0,e^{x}w(n,s)\in\{n,n+1,\ldots\}\}$ and $\widetilde{E}=[0,\infty)\times\rz$. Set $k:=k(s,x,n):=e^xw(n,s)\in\{n,n+1,\ldots\}$ for $(s,x)\in\widetilde{E}_n$ and $n\in\nz$. The semigroups $(\widetilde{T}^{(n)}_t)_{t\ge0}$ and $(\widetilde{T}_t)_{t\ge0}$ of $\widetilde{Y}^{(n)}$ and $\widetilde{Y}$ are given by
	\begin{eqnarray*}
		\widetilde{T}_t^{(n)}f(s,x)& := &\me(f(s+t,Y_{s+t}^{(n)})|Y_s^{(n)}=x)\ =\ \me(f(s+t,\log L_t^{(k)}-\log w(n,s+t))\\& = &\me(f(s+t,\log(w(k,t)/w(n,s+t))+Y_t^{(k)})),\qquad (s,y)\in\widetilde{E}_{n},
	\end{eqnarray*}
	and
	\[
		\widetilde{T}_tf(s,x)\ :=\ \me(f(s+t,Y_{s+t})|Y_s=x)\ =\ \me(f(s+t,e^{\kappa t}x+Y_t)),\qquad (s,x)\in\widetilde{E},
	\]
	for $f\in B(\widetilde{E}),t\ge 0$ and $n\in\nz$. Fix $t>0$ and first let $f\in B(\widetilde{E})$ be of the form $f(s,x)=g(s)h(x)$, $(s,x)\in\widetilde{E}$, where $g\in B([0,\infty))$ and $h\in C_c(\rz)$. Clearly, $\widetilde{T}^{(n)}_tf(s,x)=g(s+t)\me(h(\log(w(k,t)/w(n,s+t))+Y_t^{(k)}))$, $(s,x)\in \widetilde{E}_n,n\in\nz$, and $\widetilde{T}_tf(s,x)=g(s+t)\me(h(e^{\kappa t}x+Y_t))$, $(s,x)\in\widetilde{E}$. If we are able to show that
	\begin{equation}
		\lim_{n\to\infty}\sup_{(s,x)\in\widetilde{E}_n}\big|\me\big(h(\log(w(k,t)/w(n,s+t))+Y_t^{(k)})\big)-\me(h(e^{bt}x+Y_t))\big|\ =\ 0,
	\label{eq_proof_fix_line_local_06}
	\end{equation}	
	then
	\begin{equation}
		\lim_{n\to\infty}\sup_{(s,x)\in\widetilde{E}_n}|\widetilde{T}_t^{(n)}f(s,x)-\widetilde{T}_tf(s,x)|\ =\ 0.
	\label{eq_Thm_4_local_4}	
	\end{equation}
	The algebra of functions $f\in B(\widetilde{E})$ of the form $f(s,x)=\sum_{i=1}^{l}g_i(s)h_i(x)$, $(s,x)\in\widetilde{E}$, where $l\in\nz,g_i\in B([0,\infty))$ and $h_i\in C_c(\rz)$, separates points and vanishes nowhere. According to the Stone--Weierstrass theorem for locally compact spaces (see e.g. \cite{debranges59}) it is a dense subset of $B(\widetilde{E})$ such that (\ref{eq_Thm_4_local_4}) holds for $f\in B(\widetilde{E})$. \cite[IV, Theorem 2.11]{ethierkurtz} states that
	$\widetilde{Y}^{(n)}\to \widetilde{Y}$ in $D_{\widetilde{E}}[0,\infty)$, hence $Y^{(n)}\to Y$ in $D_{\rz}[0,\infty)$ as $n\to\infty$. It remains to verify (\ref{eq_proof_fix_line_local_06}).
	
	From
	\[
		s+t\ =\ \int_x^{w(x,s)}\frac{{\rm d}u}{\gamma(u)}\ +\ \int_{w(x,s)}^{w(w(x,s),t)}\frac{{\rm d}u}{\gamma(u)}\ =\ \int_x^{w(w(x,s),t)}\frac{{\rm d}u}{\gamma(u)}
	\]
	it follows that $w(x,s+t)=w(w(x,s),t)$ for $(s,x)\in\widetilde{E}$. By Proposition \ref{prop_normalizing_fix_line}, there exists a slowly varying function $L_t^{\#}:[1,\infty)\to(0,\infty)$ such that $w(x,t)=x^{e^{\kappa t}}L_t^{\#}(x)$ for $x\ge1$. Applying Proposition \ref{prop_app_slow_variation} to the right-side of
	\[
		\frac{w(k,t)}{w(n,s+t)}\ =\ \frac{w(e^xw(n,s),t)}{w(w(n,s),t)}\ =\ e^{e^{\kappa t}x}\frac{L_t^{\#}(e^xw(n,s))}{L_t^{\#}(w(n,s))},\qquad (s,x)\in\widetilde{E}_n,n\in\nz,
	\]
	provides
	\[
		\lim_{x\to\infty}\inf_{s:(s,x)\in\widetilde{E}_n,n\in\nz}\log\frac{w(k,t)}{w(n,s+t)}\ =\ \infty,\quad \lim_{x\to-\infty}\sup_{s:(s,x)\in\widetilde{E}_n,n\in\nz}\log\frac{w(k,t)}{w(n,s+t)}\ =\ -\infty.
	\]
	The family $\{Y_t^{(k)}:k\in\nz\}$ is tight, due to the convergence $Y_t^{(k)}\to Y_t$ in distribution as $k\to\infty$ and Prokhorov's theorem. By dominated convergence and since $h$ has compact support,
	\begin{eqnarray}	\label{eq_proof_fix_line_local_03}
		&&\lim_{x\to\infty}\sup_{s:(s,x)\in\widetilde{E}_n,n\in\nz}\big|\me\big(h(\log(w(k,t)/w(n,s+t))+Y_t^{(k)})\big)\big|\\&&~~~~~\ =\ 	\lim_{x\to-\infty}\sup_{s:(s,x)\in\widetilde{E}_n,n\in\nz}\big|\me\big(h(\log(w(k,t)/w(n,s+t))+Y_t^{(k)})\big)\big|\ =\ 0,\nonumber
	\end{eqnarray}
	such as
	\begin{equation}
		\lim_{x\to\infty}\me(h(e^{\kappa t}x+Y_t))\ =\ \lim_{x\to-\infty}\me(h(e^{\kappa t}x+Y_t))\ =\ 0.
	\label{eq_proof_fix_line_local_04}
	\end{equation}
	For any compact interval $K\subset\rz$ we have that, by the uniform convergence theorem for slowly varying functions \cite[Theorem 1.5.2]{binghamgoldieteugels},	
	\[
		\lim_{n\to\infty}\sup_{(s,x)\in\widetilde{E}_n,x\in K}\bigg| \log\frac{w(k,t)}{w(n,s+t)}-e^{\kappa t}x\bigg|\ =\ \lim_{n\to\infty}\sup_{(s,x)\in\widetilde{E}_n,x\in K}\bigg|\log \frac{L_t^{\#}(e^xw(n,s))}{L_t^{\#}(w(n,s))}\bigg|\ =\ 0.
	\]
	The function $h$ is uniformly continuous. Note that $\lim_{n\to\infty}\inf_{(s,x)\in\widetilde{E}_n,x\in K}k(s,x,n)=\infty$. From the convergence $Y_t^{(k)}\to Y_t$ in distribution as $k\to\infty$ it hence follows that
	\begin{equation}
		\lim_{n\to\infty}\sup_{(s,x)\in\widetilde{E}_n,x\in K}\big|\me\big(h(\log(w(k,t)/w(n,s+t))+Y_t^{(k)})\big)-\me(h(e^{\kappa t}x+Y_t))\big|\ =\ 0.
	\label{eq_proof_fix_line_local_05}
	\end{equation}
	Finally, Eqs. (\ref{eq_proof_fix_line_local_03}), (\ref{eq_proof_fix_line_local_04}) and (\ref{eq_proof_fix_line_local_05}) imply (\ref{eq_proof_fix_line_local_06}). The proof is complete.\hfill$\Box$
\end{proof}

\subsection{Appendix}
We collect some fundamental results concerning the model described in Section \ref{sec_block_counting} involving an infinite number of urns. Let $u\in\Delta$. Recall that $X_i(n,u)$ denotes the number of balls in urn $i\in\nz_0$ after $n$ balls have been allocated. Let $K(n,u):=\sum_{i\ge 1}1_{\{X_i(n,u)>0\}}$ denote the number of occupied urns (disregarding urn $0$). The following law of large numbers result holds.
\begin{lemma}
   For all $u\in\Delta$, $K(n,u)/\me(K(n,u))\to 1$ almost surely as
   $n\to\infty$.
   \label{lem_conv_K}
\end{lemma}
\begin{proof}
    We proceed as in the proof of \cite[Theorem 1]{gnedinhansenpitman07}. Fix $u\in\Delta$ and write $K(n):=K(n,u)$ for convenience. Define $\Phi:[0,\infty)\to[0,\infty)$ via $\Phi(x):=\sum_{i\ge 1}(1-(1-u_i)^x)$, $x\ge 0$. Note that $\Phi(0)=0$, $\Phi(1)=|u|\le 1$ and $\me(K(n))=\Phi(n)$, $n\in\nz$. The function $\Phi$ is nondecreasing, concave and differentiable on $(0,\infty)$ with derivative $\Phi'(x)=\sum_{i\ge 1}(1-u_i)^x(-\log(1-u_i))$. In particular, for all $x\ge 1$,
    $\Phi'(x)\le\Phi'(1)=\sum_{i\ge 1}(1-u_i)(-\log(1-u_i))\le\sum_{i\ge 1}u_i\le 1$.
    Thus, for each $m\in\nz$ there exists $n_m\in\nz$ such that $m^2\le\Phi(n_m)\le m^2+1$. Tschebyscheff's inequality together with
    ${\rm Var}(K(n))\le\Phi(2n)-\Phi(n)\le\Phi(n)$ yields
    \[
    \pr\bigg(\bigg|\frac{K(n_m)}{\Phi(n_m)}-1\bigg|\ge\varepsilon\bigg)
    \ \le\ \frac{{\rm Var}(K(n_m))}{\varepsilon^2(\Phi(n_m))^2}
    \ \le\ \frac{1}{\varepsilon^2\Phi(n_m)}
    \ \le\ \frac{1}{\varepsilon^2 m^2}
    \]
    for all $m\in\nz$ and $\varepsilon>0$. Thus, $\sum_{m\ge 1}\pr(|K(n_m)/\Phi(n_m)-1|\ge\varepsilon)<\infty$ for all $\varepsilon>0$. By the Borel--Cantelli lemma it follows that $K(n_m)/\Phi(n_m)\to 1$ almost surely as $m\to\infty$.

    For $n\in\nz$ with $n_m\le n\le n_{m+1}$ the monotonicity inequalities $K(n_m)\le K(n)\le K(n_{m+1})$ and $\Phi(n_m)\le\Phi(n)\le\Phi(n_{m+1})$ hold, which allows to sandwich the fraction $K(n)/\Phi(n)$ via
    \[
    \frac{K(n_m)}{\Phi(n_{m+1})}\ \le\ \frac{K(n)}{\Phi(n)}\ \le\ \frac{K(n_{m+1})}{\Phi(n_m)},
    \]
    where both sides converge to $1$ almost surely, since $\Phi(n_m)/\Phi(n_{m+1})\to 1$.\hfill$\Box$
\end{proof}
The following two results deal with the random variables $Y(n,u):=X_0(n,u)+K(n,u)$ defined in (\ref{eq_y}). Lemma \ref{lem_conv_Y} concerns
the limiting behavior of $Y(n,u)/n$ as $n\to\infty$.
\begin{lemma}
    For all $u=(u_1,u_2,\ldots)\in\Delta$, $Y(n,u)/n\to u_0$ almost surely as $n\to\infty$,
    where $u_0:=1-|u|:=1-\sum_{i\ge 1}u_i$.
    \label{lem_conv_Y}
\end{lemma}
\begin{proof}
    Fix $u\in\Delta$. We have $Y(n,u)=X_0(n,u)+K(n,u)$, $n\in\nz$. Clearly, $X_0(n,u)/n\to u_0$ almost surely as $n\to\infty$, since $X_0(n,u)$ has a binomial distribution with parameters $n$ and $u_0$. By Lemma \ref{lem_conv_K},
    $K(n,u)/\me(K(n,u))\to 1$ almost surely as $n\to\infty$. Moreover,
    \[
    \frac{\me(K(n,u))}{n}\ =\ \sum_{i\ge 1}\frac{1-(1-u_i)^n}{n}\ \to\ 0
    \]
    as $n\to\infty$ by dominated convergence, since $(1-(1-u_i)^n)/n\le 1/n\to 0$ and $(1-(1-u_i)^n)/n\le u_i$, where the dominating map $i\mapsto u_i$ is integrable with respect to the counting measure on $\nz$. Thus,
    \[
    \frac{K(n,u)}{n}\ =\ \frac{K(n,u)}{\me(K(n,u))}\frac{\me(K(n,u))}{n}\ \to\ 1\cdot 0\ =\ 0
    \]
    almost surely as $n\to\infty$. Therefore, $Y(n,u)/n\to u_0$ almost surely as $n\to\infty$.\hfill$\Box$
\end{proof}
The following result (Lemma \ref{lem_bound_Y}) is used in the proof of the main convergence theorem (Theorem \ref{thm_main}). It presents bounds for particular moments of the random variable $Y(n,u)$ defined in (\ref{eq_y}). Lemma 18 of \cite{limic10} provides similar bounds.
\begin{lemma}
	There exist constants $D_1,D_2\in\rz$ such that, for all $u\in\Delta$,
	\[
    \sup_{n\in\nz}\me\bigg(\bigg(\frac{Y(n,u)}{n}-(1-|u|)\bigg)^2\bigg)\ \le\ D_1|u|^2
    \quad\mbox{and}\quad
    \sup_{n\in\nz}\me\bigg(\bigg(\frac{Y(n,u)}{n}-1\bigg)^2\bigg)\ \le\ D_2|u|^2.
	\]
	\label{lem_bound_Y}
\end{lemma}
\begin{proof}
	Fix $n\in\nz$ and $u\in\Delta$. Define $u_0:=1-|u|$. We omit the parameter $(n,u)$ and write (\ref{eq_y}) as $Y=X_0+K$, where $K:=K(n,u):=\sum_{i\ge 1}1_{\{X_i(n,u)>0\}}$ denotes the number of occupied urns (disregarding urn $0$). Furthermore, define  $\widetilde{Y}:=Y/n-u_0=X_0/n-u_0+K/n$. Calculations that are similar to the following (but come from a different motivation) are carried out in the proof of \cite[Lemma 6.1]{moehle21}. The formulas for $\me(K)$ and $\me(X_0K)$ can be found there. We have
	\[
		\widetilde{Y}^2\ =\ \bigg(\frac{X_0}{n}-u_0+\frac{K}{n}\bigg)^2\ =\ \bigg(\frac{X_0}{n}-u_0\bigg)^2\ +\ \frac{2X_0K}{n^2}\ -\ \frac{2u_0K}{n}\ +\ \frac{K^2}{n^2}.
	\]
	Recall that $X_0$ has a binomial distribution with parameters $n$ and $u_0$. In particular, $\me(X_0)=nu_0$ and $\me((X_0/n-u_0)^2)=n^{-2}\mathrm{Var}(X_0)=|u|(1-|u|)/n$. Together with $K^2=K+\sum_{i\ne j}1_{\{X_iX_j>0\}}$ it follows that
	\[ \me(\widetilde{Y}^2)\ =\ \frac{u_0(1-u_0)}{n}\ +\ \frac{2\me(X_0K)}{n^2}\ -\ \frac{2u_0\me(K)}{n}\ +\ \frac{\me(K)}{n^2}\ +\ \frac{1}{n^2}\sum_{i\ne j}\pr(X_iX_j>0).
\]
Adding and subtracting $u_0(1-u_0)/n=|u|(1-|u|)/n$ leads to
\[ \me(\widetilde{Y}^2)\ =\ 2\bigg(\frac{\me(X_0K)}{n^2}-\frac{u_0\me(K)}{n}+\frac{u_0(1-u_0)}{n}\bigg)\ +\ \frac{1}{n}\bigg(\frac{\me(K)}{n}-|u|\bigg)\ +\ \frac{|u|^2}{n}\ +\ \frac{1}{n^2}\sum_{i\ne j}\pr(X_iX_j>0).
	\]
	We have $\me(K)=\sum_{i\ge 1}(1-(1-u_i)^n)$ and $\me(X_0K)=nu_0\sum_{i\ge 1}(1-(1-u_i)^{n-1})$.
	Moreover, by Bernoulli's inequality, $1-(1-u_i)^{n-1}\le (n-1)u_i$ for $i\in\nz$. We conclude that
	\[	 \frac{\me(X_0K)}{n^2}\ -\ \frac{u_0\me(K)}{n}\ +\ \frac{u_0(1-u_0)}{n}\ =\ \frac{u_0}{n}\sum_{i\ge 1}u_i(1-(1-u_i)^{n-1})\ \le\    \frac{(n-1)u_0}{n}(u,u)\ \le\ |u|^2.
	\]
	Also, $n^{-1}\me(K)-|u|=n^{-1}\sum_{i\ge 1}(1-(1-u_i)^n-nu_i)\le 0$. From the generalized Bernoulli inequality $1-(1-u_i)^n-(1-u_j)^n+(1-u_i-u_j)^n\le n(n-1)u_iu_j$, $i,j\in\nz$, it follows that
	\[
		\frac{1}{n^2}\sum_{i\ne j}\pr(X_iX_j>0)\ =\ \frac{1}{n^2}\sum_{i\ne j}(1-(1-u_i)^n-(1-u_j)^n+(1-u_i-u_j)^n)\ \le\ \sum_{i,j\ge 1}u_iu_j\ =\ |u|^2.
	\]
    Collecting all bounds it follows that $\me(\widetilde{Y}^2)$ is bounded by $4|u|^2$, which shows that the first claim holds with $D_1:=4$. Concerning the second claim, note that
	\[
		0\ \le\ \me((Y/n-1)^2)\ =\ \me((\widetilde{Y}-|u|)^2) \ =\ \me(\widetilde{Y}^2)\ -\ 2|u|\me(\widetilde{Y})\ +\ |u|^2
	\ \le\ \me(\widetilde{Y}^2)\ +\ |u|^2,
	\]
	since $\me(\widetilde{Y})=n^{-1}\me(K)\ge 0$, showing that we can choose $D_2:=D_1+1=5$.
\hfill$\Box$\end{proof}
The following result is needed in the proof of Theorem \ref{thm_fix_line}.
\begin{proposition}
	Let $\alpha>0$ and $L:[1,\infty)\to(0,\infty)$ be slowly varying with $0<\inf_{y\in[1,K]}L(y)\le\sup_{y\in[1,K]}L(y)<\infty$ for any $K>1$. Then $\lim_{x\to\infty}\inf_{y\ge x^{-1}\vee1}x^{\alpha}L(xy)/L(y)=\infty$ and $\lim_{x\to0}\sup_{y\ge x^{-1}\vee1}x^{\alpha}L(xy)/L(y)=0$.
\label{prop_app_slow_variation}
\end{proposition}

\begin{proof}
	By the representation theorem for slowly varying functions \cite[Theorem 1.3.1]{binghamgoldieteugels}, there exist functions $\delta:[1,\infty)\to(0,\infty)$ and $\varepsilon:[1,\infty)\to\rz$ with $\lim_{x\to\infty}\delta(x)=:d\in(0,\infty)$ and $\lim_{x\to\infty}\varepsilon(x)=0$ such that $L(x)=\delta(x)\exp(\int_{1}^{x}\frac{\varepsilon(u)}{u}{\rm d}u)$, $x\ge1$. Furthermore we can choose $\varepsilon$ such that $\|\varepsilon\|\le\alpha/2$, if $\delta$ is adapted accordingly. By the additional boundary assumption for $L$ and the convergence of $\delta$ to $d\in(0,\infty)$, $0<\inf_{y\ge1}\delta(y)\le\sup_{y\ge1}\delta(y)<\infty$. Thus, $0<\inf_{y\ge x^{-1}\vee1}\delta(xy)/\delta(y)\le \sup_{y\ge x^{-1}\vee1}\delta(xy)/\delta(y)<\infty$. From
	\[
		x^{\alpha}\frac{L(xy)}{L(y)}\ =\ \frac{\delta(xy)}{\delta(y)}\exp\bigg(\int_{1}^{x}\frac{\alpha}{u}\,{\rm d}u+\int_{y}^{xy}\frac{\varepsilon(u)}{u}\,{\rm d}u\bigg)\ =\ \frac{\delta(xy)}{\delta(y)}\exp\bigg(\int_{1}^{x}\frac{\alpha+\varepsilon(uy)}{u}\,{\rm d}u\bigg)
	\]
	it follows that
	\[
		\inf_{y\ge x^{-1}\vee1}x^{\alpha}\frac{L(xy)}{L(y)}\ \ge\ \inf_{y\ge x^{-1}\vee1}\frac{\delta(xy)}{\delta(y)}\exp\bigg(\int_{1}^{x}\frac{\alpha}{2u}\,{\rm d}u\bigg)\ =\ x^{\alpha/2}\inf_{y\ge x^{-1}\vee1}\frac{\delta(xy)}{\delta(y)}\ \to\ \infty
	\]
	as $x\to\infty$ and
	\[
		\sup_{y\ge x^{-1}\vee1}x^{\alpha}\frac{L(xy)}{L(y)}\ \le\  x^{\alpha/2}\sup_{y\ge x^{-1}\vee1}\frac{\delta(xy)}{\delta(y)}\ \to\ 0
	\]
	as $x\to0$.
\hfill$\Box$\end{proof}

\begin{remark}
	The function $L_t^{\#}:[1,\infty)\to(0,\infty)$, defined via $L_t^{\#}(x):=w(x,t)/x^{e^{\kappa t}}$, $x\ge1$, is slowly varying. Due to $w(x,t)\ge x$, it holds that $L_t^{\#}(x)\ge x^{1-e^{\kappa t}}$ for $x\ge1$ and $t\ge0$, and since $L_t^{\#}$ is continuous, Proposition \ref{prop_app_slow_variation} applies.
\end{remark}


\end{document}